\documentclass[11pt,reqno]{amsart}
\usepackage{fullpage}
\usepackage{amsmath,amsfonts, amsthm, amssymb,bbm}
\usepackage{setspace} 
\usepackage{xcolor}
\usepackage{enumerate}
\usepackage{tikz}
\usepackage{tikz-cd}
\usetikzlibrary{tikzmark}
\usepackage{multirow}
\usepackage[normalem]{ulem}
\usepackage{comment}
\usepackage{tcolorbox}
\usepackage{hyperref}
\usepackage{subfig}
\usepackage{graphicx}
\usepackage{blkarray}

\theoremstyle{plain}
\newtheorem{theorem}{Theorem}[section]
\newtheorem{prop}[theorem]{Proposition}
\newtheorem{lem}[theorem]{Lemma} 
\newtheorem{coro}[theorem]{Corollary}

\newtheorem{claim}[theorem]{Claim}

\theoremstyle{definition}
\newtheorem{definition}[theorem]{Definition}
\newtheorem{example}[theorem]{Example}
\newtheorem{remark}[theorem]{Remark}
\newtheorem{question}[theorem]{Question}

\newcommand{\ii}{{\mathrm{i}}}
\newcommand{\dd}{{\mathrm{d}}}
\newcommand{\ee}{{\mathrm{e}}}
\newcommand{\one}{{\mathbbm{1}}}
\newcommand{\ZZ}{{\mathbb{Z}}}

\newcommand{\mA}{{\mathcal{A}}}
\newcommand{\mP}{{\mathcal{P}}}

\newcommand{\exend}{\hfill $\Diamond$}

\def\RO{\textbf{(R1)}}
\def\RT{\textbf{(R2)}}

\def\R{\mathbb{R}}
\def\N{\mathbb{N}}
\def\Z{\mathbb{Z}}
\def\T{\mathcal{T}}

\newcommand{\sub}{{\mathcal{S}}}
\newcommand{\digits}{{\mathcal D}}
\def\vecj{\vec{j}}
\def\Gdual{\widehat{G}}

\newcommand{\subs}[1]{\mathcal{S}^{#1}}
\newcommand{\digit}[1]{\mathcal{D}^{(#1)}}

\def\alet{\mathsf{a}}
\def\blet{\mathsf{b}}

\def\dlet{\mathsf{d}}
\def\tiling{\mathcal{T}}
\def\seq{\tau}
\def\vecd{{\vec{d}}}

\def\seqsp{\Sigma}
\def\ssp{{{\Sigma_\sq}}}

\def\sat{{\Sigma_\fractile}}

\def\dig{\pi^{ }_{\digits}}
\def\spin{\pi^{ }_{G}}
\def\spinfac{\pi_{G/\text{ker}(\chi)}}

\def\sq{{\mathfrak{u}}}
\def\fractile{{\mathfrak{t}}}

\def\stri{\sub_{\textnormal{tri}}}

\def\odom{\mathcal{O}}
\newcommand{\addr}[1]{{\mathbf{#1}}}
\newcommand{\odigit}[1]{\vec{#1}}

\def\haar{\mu^{ }_{\textnormal{H}}}
\def\leb{\mu^{ }_{\textnormal{L}}}

\begin{document}

\title{Spectral theory of spin substitutions}

 \author{Natalie Priebe Frank}
 \address{Department of Mathematics and Statistics, Vassar College, \newline \hspace*{\parindent}Box 248, Poughkeepsie, NY  12604, USA}
 \email{nafrank@vassar.edu}

 \author{Neil Ma\~nibo}
\address{Fakult\"at f\"ur Mathematik, Universit\"at Bielefeld, \newline
\hspace*{\parindent}Postfach 100131, 33501 Bielefeld, Germany}
\email{cmanibo@math.uni-bielefeld.de }

\begin{abstract}
We introduce qubit substitutions in $\Z^m$, which have non-rectangular domains based on an endomorphism $Q$ of $\Z^m$ and a set $\digits$ of coset representatives of $\Z^m/Q\Z^m$. 
We then focus on a specific family of qubit substitutions which we call spin substitutions, whose combinatorial definition requires a finite abelian group $G$ as its spin group. 
We investigate the spectral theory of the underlying subshift $(\seqsp,\Z^m)$. Under certain assumptions, we show that it is measure-theoretically isomorphic to a group extension of an $m$-dimensional odometer, which induces a complete decomposition of the function space $L^{2}(\seqsp,\mu)$ . 
This enables one to use group characters in $\Gdual$ to derive substitutive factors and 
carry out  a spectral analysis on specific subspaces.
We provide general sufficient criteria for the existence of pure point, absolutely continuous and singular continuous spectral measures, together with some bounds on their spectral multiplicity. 
\end{abstract}

\keywords{dynamical spectrum, diffraction, substitution tilings, group extensions}

\subjclass[2010]{37B10, 37A30, 52C23, 42A16}

\date{\today}
\maketitle

\section{Introduction}

In this work, we define and develop the theory of \emph{qubit substitutions},  which are a type of multidimensional morphisms on a finite alphabet $\mA$ wherein the geometry is based on digit tilings; see \cite{Vince2,Vince}. 
Given an expanding endomorphism $Q$ of $\Z^m$ and a full set $\digits$ (our {\em digit set}) of coset representatives  of $\Z^m/Q\Z^m$  we can partition $\Z^m$ into subsets $Q \vecj + \digits$. There are already interesting questions on `tiling the integers' (see numerous references from \cite{CM,GT}), but our construction is a type of discrete finite automaton instead. 
The choice of a digit set $\digits$ completely determines a sequence $\left\{\digits^{(n)}\right\}$ of coset representatives for $\Z^m/Q^{n}\Z^m$ for all $n\in \N$, which is suitable in building the support of the level-$n$ supertiles of a qubit substitution; see \cite{Cabezas} for a recent account which focuses on algebraic invariants. This generalizes the hierarchical structure present in constant-length substitutions in one dimension and block substitutions in higher dimensions, compare \cite{Queffelec2, Frank3, Bartlett, BG2}. 

Once the issues regarding the underlying geometry are settled, it becomes clear that one can use the same techniques in constructing analogous variants in the qubit setting (e.g. substitutions with coincidences, bijective substitutions, etc.). For the rest of the paper, we then focus on a method for defining {\em qubit spin substitutions}, which is inspired by the construction of the Rudin--Shapiro sequence as an iterated morphism. 
These substitutions are special types of qubit substitutions having additional invariance properties from the defining spin group $G$.
The resulting subshifts allow a relatively complete  analysis that generalizes known families algebraically, geometrically and spectrally; see \cite{Queffelec1,AL-2,Frank1,AL,CGS}. The construction works in any dimension, allows the existence of disconnected supertiles and holds for any choice of (finite) abelian group. Under additional assumptions, a subshift with the same spectral features can be defined even when the spin group is no longer finite, but is still compact, yielding new and interesting phenomena.

We will provide details and motivating examples, but here is the basic idea. In addition to $Q$ and $\digits$, we take a finite abelian group $G$ (our {\em spin group})  and a $|\digits| \times |\digits|$ matrix $W$ over $G$. The matrix $W$, in conjunction with the decomposition of $\vecj$ base $(Q,\digits)$, determines how to allocate a spin to $\vecj$. The resulting rule then maps an element of the alphabet $\mA=G\times\digits$ to a finite word over $\mA$ supported on $\digits$, which can be unambiguously iterated under some assumptions on the digit set $\digits$ and can be used to define a subshift $(\seqsp,\Z^m)$. This subshift is our main object of interest.

The paper is organized as follows. In Section~\ref{sec:qubitandspin}, we recall notions and classical results regarding digit tilings on $\Z^m$ and establish how to build a qubit substitution subshift from it. 
We then discuss certain dynamical properties of the subshift which it inherits from an associated tiling dynamical system (with an $\R^m$-action). Here it is made clear that the connectedness of the supertiles is not a prerequisite to define the subshift. 

Section \ref{sec:dynamspec} begins with a brief survey of notions in spectral theory, followed by a discussion of the group action by $G$ on the subshift. Proposition~\ref{prop: L2 decomp} provides the decomposition of $L^{2}(X,\mu)$ into subspaces $H^{\chi}$ which are in one-to-one correspondence with the characters $\chi\in \Gdual$, and Theorem~\ref{thm:main result} gives sufficient conditions for the presence of certain spectral types. We will write $\chi(W)$ for the matrix with entries $\chi(W_{ij})$. 
Under mild dynamical assumptions, Theorem~\ref{thm:main result} shows: 

\begin{quote}
\begin{itemize}
\item If $\chi$ is the trivial character, then the subspace $H^{\chi}$ is pure point. 
\item If $\frac{1}{|\digits|}\chi(W)$ is unitary, then $H^{\chi}$ is absolutely continuous. 
\item If $\chi(W)$ is rank-$1$, then $H^{\chi}$ is either pure point or purely singular continuous.
\end{itemize}
\end{quote} 

Note that these spectral results are independent of the geometry and only rely on $W$. These three statements are proved in different sections where the appropriate arguments are developed separately. 
In Section~\ref{sec:odometer}, we investigate the underlying odometer, which is a pure point factor recoverable via the trivial character. More specifically, in Theorem~\ref{thm:skew product}, we prove that a spin substitution subshift is isomorphic to a group extension and give an explicit formula for the cocycle. 
Proposition~\ref{prop: L2 decomp} then follows from this as a corollary.

The splitting of the function space enables one to handle each subspace independently. Proposition~\ref{prop:purity-skew} states that every subspace is spectrally pure.
Section~\ref{sec:ac} contains
the proof of the second statement, which relies on computing the Fourier coefficients 
of certain spectral measures explicitly. The method of proof is constructive and follows that of \cite{Frank1}.

In Section~\ref{sec:diffraction}, we establish the connection between the diffraction spectrum and the spectral measures of specific functions that we know are in $H^{\chi}$. This allows us to use spectral results in diffraction theory to obtain dynamical conclusions. Proposition~\ref{prop: factors} describes the properties of substitutive factors, which always exist for spin substitutions, and how the spectral measures they admit restrict the spectral type of $H^{\chi}$. Proposition~\ref{prop:rank1} and the discussion following it finish the proof of the third claim on singularity. For this we need the construction of the factor substitution induced by a character $\chi\in \Gdual$ (with the extra assumption of $\chi(W)$ being rank-$1$). 

When $\chi$ does not satisfy any of the conditions above, 
but $\chi(W)$ is still of full rank, 
we turn to diffraction analysis in view of the renormalization approach via Lyapunov exponents; see Section~\ref{sec:lyapunov}. To be more precise, the corresponding matrix cocycle $B$ is unitarily block-diagonalizable, where each block  $B_{\chi}$ corresponds to a character $\chi$ and its subspace $H^{\chi}$.
Proposition~\ref{prop:Lyapunov-spin-singular} provides a numerically-verifiable singularity criterion for $H^{\chi}$.

Several examples with particular spectral properties are then presented in Section~\ref{sec:examples}.

\section{Qubit and spin substitutions}\label{sec:qubitandspin}

The underlying geometry of a qubit substitution is given by a periodic `digit' tiling of $\Z^m$ that is generated in a substitutive manner \cite{Vince}. Qubit substitutions use the lattice and inflation structure of the digit tiling, rather than the canonical rectangular regions, upon which to make a substitution rule. In many cases, this method can produce subshifts even when the digit set is highly non-canonical. 

The main results appearing in this paper are for a special class of qubit substitutions we call {\em spin substitutions}. In this section, we lay out the definitions of qubit substitution dynamical systems, define spin substitutions, and provide illustrative examples.

\subsection{Qubit substitutions}\label{sec:digitqubit}
A {\bf\em digit system} $(Q, \digits)$ (see \cite{Vince}) is an expansive endomorphism $Q$ of $\Z^m$ along with a complete set of coset representatives of $\mathbb{Z}^m/Q\mathbb{Z}^m$ called a {\bf \em digit set} \footnote{This is called a {\em standard digit set} in \cite{LW-2}.}
$\digits$. Since $\Z^m = Q  \Z^m + \digits$, $\digits$ (seen as a direct Minkowski sum) is a set that covers $\Z^m$ periodically and comes with a natural rescaling map $Q$. Those properties allow digit systems to support the underlying geometry of a substitution rule in $\Z^m$ or $\R^m$.
 
\begin{definition} Let $(Q, \digits)$ be a digit system and let $\mA$ be a finite set called the  {\bf \em alphabet}. 
A {\bf \em qubit substitution} is a map $\sub = \subs1: \mA \times \digits \to \mA$.
We write $\sub(\alet)$ to mean the word $\sub(\alet,\digits)$, and we call this word a {\bf \em \textnormal{1}-supertile}. The {\bf \em n-supertiles} are defined recursively as
\[
\subs{n+1}(\alet)=\bigcup_{\vecd\in\mathcal{D}} \subs{n}(\sub(\alet,\vecd~))+Q^{n}(\vecd~).
\]
For each $n \in \N$ we call the set  $\digit{n} = Q \digit{n-1} + \digits$ the {\bf \em domain} of the $n$-supertile, where $\digit{0}:= \{0\} \subset \Z^m$. 
\end{definition}

\begin{remark}
There are two equivalent ways to view an $(n+1)$-supertile: {\bf \em fusion} and {\bf \em substitution} (see \cite{Frank2}). The substitutive approach for computing $\subs{n+1}(\alet)$ is to apply $\sub$ to each element in $\subs{n}(\alet)$. This results in the (equivalent) formulation
\[
\subs{n+1}(\alet)=\bigcup_{\vecd\in\digit{n}} \sub(\subs{n}(\alet,\vecd~))+Q(\vecd~).
\]
The fusion approach seen in our definition is that $\subs{n+1}(\alet)$ is assembled from $n$-supertiles of the same types as those in $\sub(\alet)$ translated by $Q^n$ times their corresponding vectors.
This viewpoint is needed for the majority of our proofs. \exend
\end{remark}

\begin{example} \label{ex:gasket}We present an example of a qubit substitution on two letters. Let \[
Q=\begin{pmatrix}
2 & 0 \\
0 & 2
\end{pmatrix}, \quad \digits=\{(0,0), (1,0), (0,1), (-1,-1)\},\quad \text{and}  \quad \mA = \{\alet,\blet\}.\]
 Let $\sub(\alet)$ take the four digits to $\alet, \alet, \alet$ and $\blet$ respectively, pictured at the left in Figure ~\ref{fig:Gasket} with $\alet$ and $\blet$ represented as pink and blue squares with their lower left corners at digits. (Tilings are formalized in Section~\ref{subsec:tilings}). Let $\sub(\blet)$ take the digits to $\blet, \blet, \blet$, and $\alet$, the opposite of $\sub(\alet)$. (This type of substitution is called {\em bijective}). 
Figure~\ref{fig:Gasket} shows the first three iterations of an $\alet$ under this qubit substitution. (Note that a copy of $\sub(\blet)$ appears at the lower left of the 2-supertile pictured in the center). \exend

\begin{figure}[ht]
\centering
\raisebox{1.1in}{\includegraphics[width=.09\textwidth]{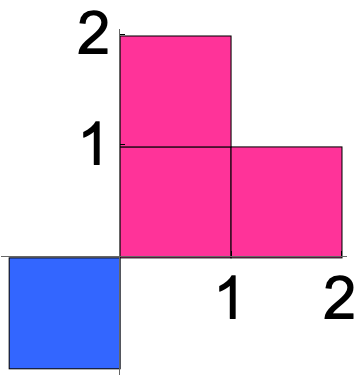}}
 \hskip 2 em
\raisebox{.7in}{\includegraphics[width=.2\textwidth]{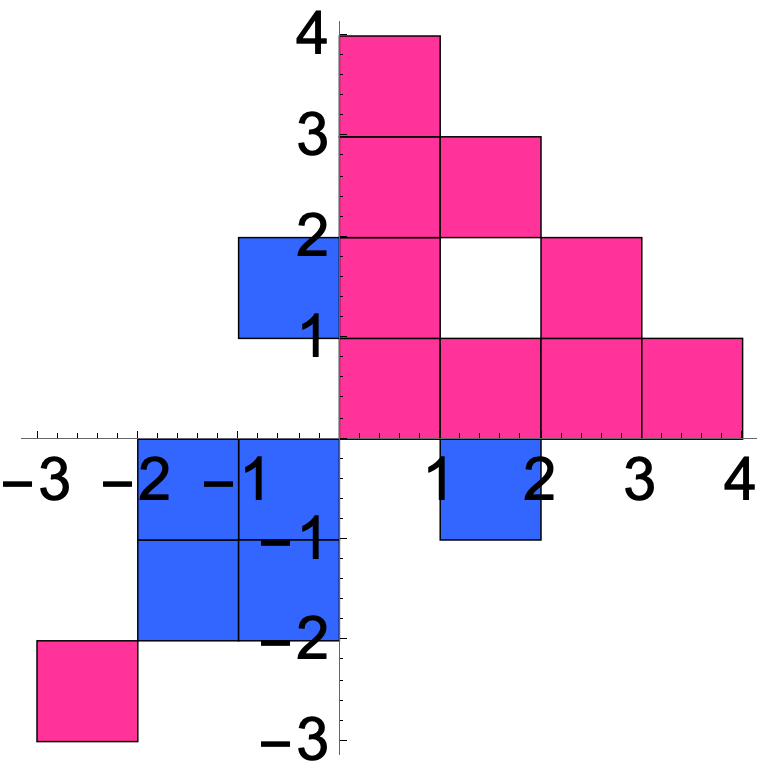}}
 \hskip 2 em
\includegraphics[width=.4\textwidth]{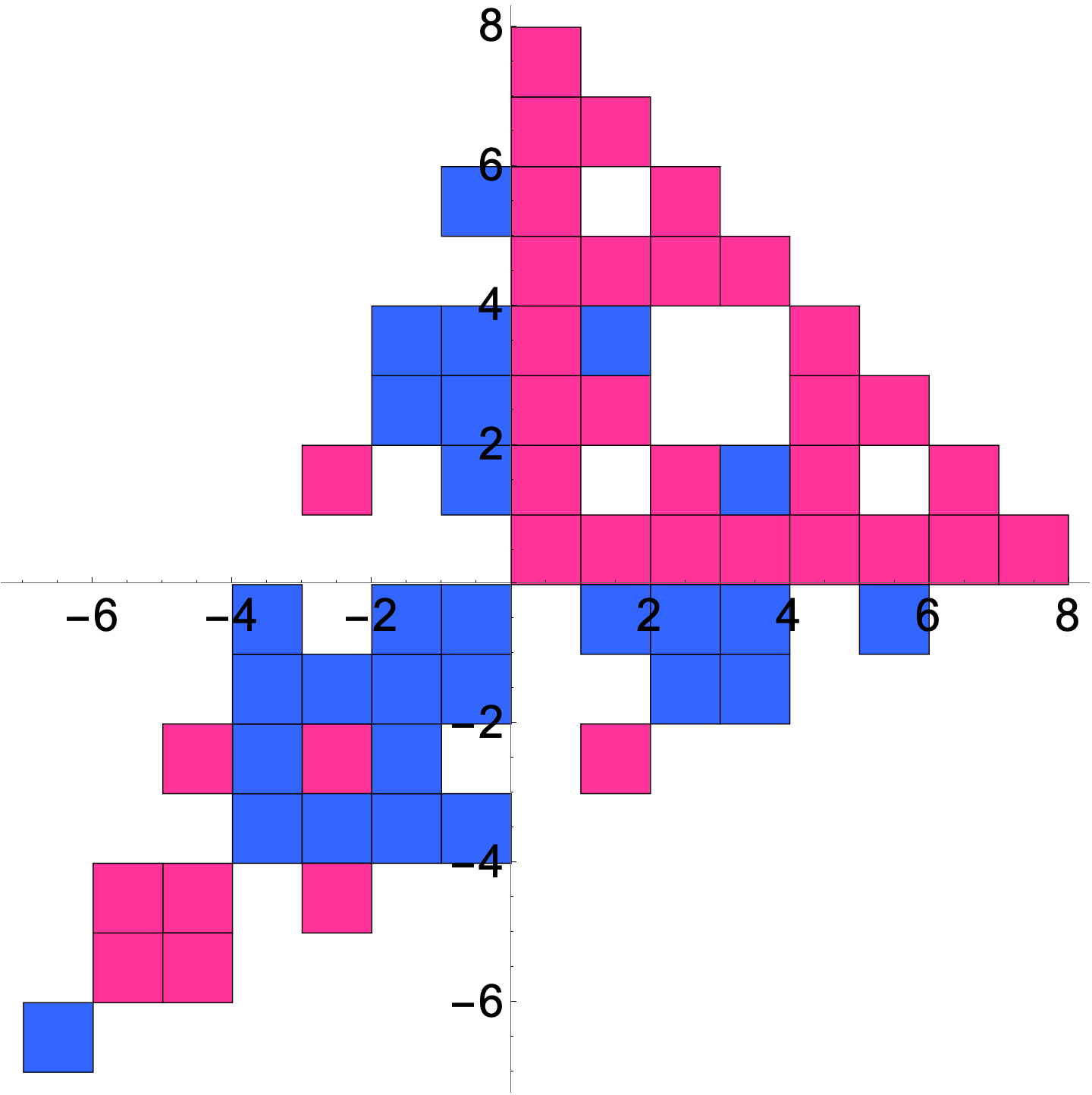}
\caption{The supertiles $\subs 1 (\alet), \subs 2  (\alet)$, and $\subs 3 (\alet)$ for Example~\ref{ex:gasket}.}\label{fig:Gasket}
\end{figure}

\end{example}

\subsection{Qubit substitution dynamical systems}

Although individual supertiles in a qubit substitution may not be connected, sometimes they contain arbitrarily large rectangular words. Whether or not a subshift exists depends entirely on whether or not they do. This in turn depends entirely on properties of its underlying digit system $(Q, \digits)$. The dynamics of the subshift depend both on the digit system and the supertile labels from $\mA$.

In the example just given, we see that any $n$-supertile contains several $2 \times 2$ words if $n \geqslant 2$. All of the iterates of these $2 \times 2$ words under the qubit substitution are therefore subwords of supertiles at some level as well. Figure \ref{fig:GasketBig} shows a later iteration of a word with domain $\{(0,0),(0,-1), (-1,-1), (-1,0)\}$.

\begin{figure}[ht]
\centering \includegraphics[width=.45\textwidth]{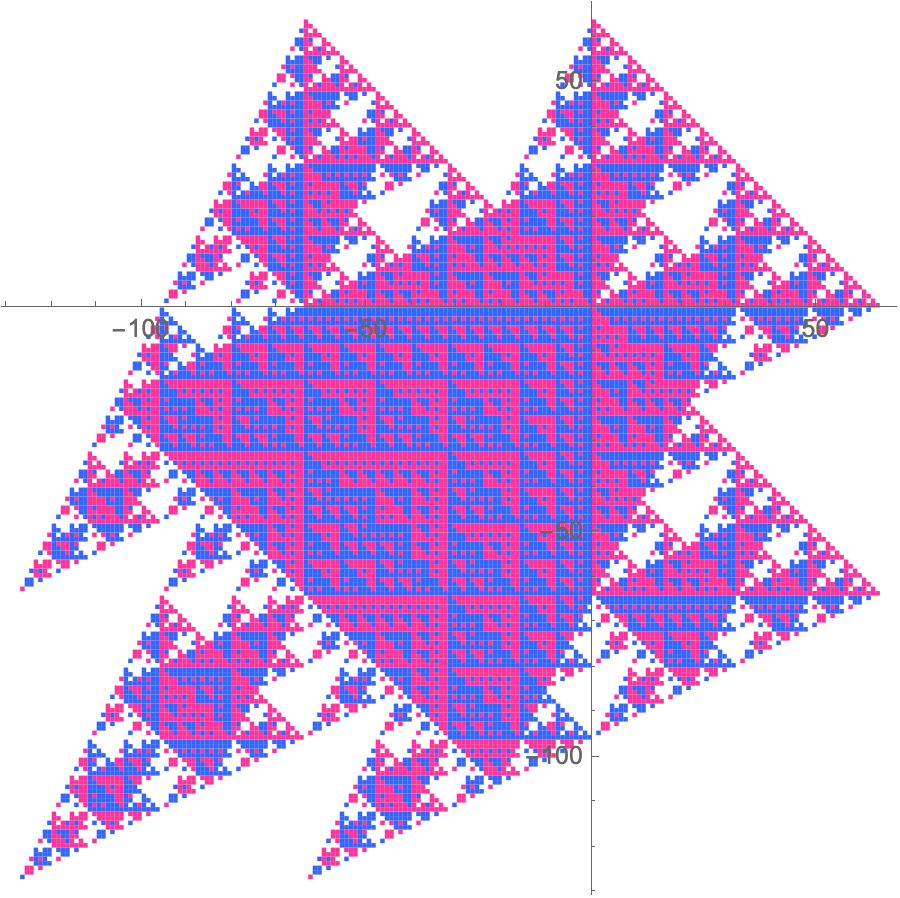}
\caption{The word $(0,0),(0,-1), (-1,-1), (-1,0) \mapsto \, \alet,\alet,\blet,\alet$ substituted six times. The presence of arbitrarily large rectangular words allows us to define a subshift of $\Z^2$ for $\sub$.}
\label{fig:GasketBig}
\end{figure}

This figure suggests that it is not necessary to have connectivity conditions on supertiles to define a subshift. Therefore, in line with standard practice, we use our supertiles as a sort of pre-language to make the following definition.  

\begin{definition}
Let $\sub$ be a qubit substitution in $\Z^m$ defined by $(Q, \digits, \mA)$, and let $R$ be a rectangular subset of $\Z^m$. A patch $P \in \mA^R$ is {\bf \em admitted by $\sub$} if there is an $N \in \N$ and an $\alet \in \mA$ such that a shift of $P$ appears in $\subs{N}(\alet)$.

An element $\seq \in \mA^{\Z^m}$ is {\bf \em admitted by $\sub$} if every rectangular subword of $\seq$ is admitted by $\sub$. If nonempty, the set
$\seqsp$ of admitted sequences in $\mA^{\Z^m}$ under the action of the shift is the {\bf \em substitution subshift $(\seqsp,\Z^m)$}.
\end{definition}

Sufficient conditions for a substitution subshift to be nonempty are not undertaken in this work but can be checked in examples.  Geometric and topological questions surrounding digit systems in general are quite subtle (see \cite{GroHaas,LW,LW-2,Vince2}), complicating our existence question.

To determine when qubit substitution dynamical systems are minimal, and to understand the nature of their invariant measures, it is necessary to delve into some of these topological questions. Our arguments will rely on existing results from tiling dynamical systems theory, so we review these briefly as well.

\subsubsection{Digit tiles and qubit tiling systems.}
\label{subsec:tilings}

We define a {\bf \em tile in $\R^m$} to be a pair of the form $(C, \alet)$, where the {\em support} $C$ is a compact subset of $\R^m$ that is the closure of its interior, and the {\em type} $\alet$ is in some finite alphabet $\mA$. Two tiles are {\em equivalent} if their  supports are translates of one another and they are the same type. A finite set $\mP$ of nonequivalent tiles on a given alphabet $\mA$ is called a {\bf \em prototile set}. Copies of prototiles are used to form tilings of $\R^m$. 

Suppose $\tiling$ is a set of tiles, all of which are equivalent to tiles in $\mP$. We call $\tiling$ a {\bf \em tiling} if the interiors of any two nonequivalent tiles are disjoint and the supports' union is $\R^m$. 
If $X$ is a set of tilings that is invariant under translation by $G=\Z^m$ or $\R^m$, and is closed under the local topology, we call $(X, G)$ a {\bf \em tiling dynamical system}.

Every digit system  $(Q, \digits)$ of the type we consider gives rise to an iterated function system with contraction maps $c_\vecd(x) = Q^{-1}(x + \vecd)$ for each $\vecd \in \digits$. The attractor is the {\bf \em digit tile}
\begin{equation}\label{def:digittile}
\fractile = \left\{\sum_{k = 1}^\infty Q^{-k} \vecd_k \, | \, \vecd_k \in \digits \right\} = \lim_{k \to \infty} Q^{-k}\digits^{(k)}.
\end{equation}
Its existence and properties are explored for broad classes of digit systems, and the results stated here are from \cite{GroHaas,LW,LW-2,Vince2} and summarized in the survey \cite{Vince}. 

The one-tile prototile set $\{\fractile\}$ can form tilings of $\R^m$ in two possibly nonequivalent ways. A `self-replicating tiling' arises from the fact that $Q\fractile = \bigcup_{\vec{d} \in \digits} \fractile + \vec{d}$. It can be shown \cite{Solomyak} that there is a tiling of $\R^m$ that is invariant under multiplication by $Q$ followed by subdivision.
 
The digit tile can also tile $\R^m$ using a lattice $L$ in the sense that 
\[ \left\{ \fractile + \vec{x}, \vec{x} \in L \right\}
\]
covers $\R^m$ and the tiles intersect only at boundaries. There are situations where $L$ is a proper sublattice of $\Z^m$
 and we need to avoid these situations.  The relevant properties needed in what follows are summarized next.

\begin{prop}(see \cite{GroHaas,LW,LW-2,Vince2})
Suppose $Q$ is an expanding endomorphism of $\Z^m$ and $\digits$ is a complete set of coset representatives of $\Z^m / Q\Z^m$. Then
\begin{itemize}
\item $\fractile$ is the closure of its interior,
\item the Lebesgue measure $\leb(\fractile)$ is a positive integer, and
\item $\leb(\fractile) = 1$ if and only if $\fractile + \Z^m$ forms a tiling of $\R^m$.  \qed
\end{itemize}
\end{prop}

\begin{definition}
We call $(Q, \digits)$ a {\bf \em unit digit system} if its digit tile has measure 1.
\end{definition}

Suppose $(\seqsp,\Z^m) $ is a nonempty subshift given by a qubit substitution over a unit digit system $(Q, \digits)$. We construct two tiling spaces for $\seqsp$ that are topologically conjugate as dynamical systems over $\R^m$.
 
Let $\sq = [0,1]^m$ be the unit cube and let $\fractile$ be the digit tile for $(Q,\digits)$. Without loss of generality, we may assume that $0 \in \digits$ so that $0 $ is in both $\sq$ and $\fractile$. Let $\mA_\sq = \{(\sq, \alet), \alet \in \mA\}$ and $\mA_\fractile = \{(\fractile, \alet), \alet \in \mA\}$  be two  prototile sets, the former for constructing a tiling space $\ssp$ and the latter for constructing $\sat$. For the substitution of Example~\ref{ex:gasket} we have two fractal prototiles ({\em fractiles}) that look like this:

\centerline{\includegraphics[width=.15\textwidth]{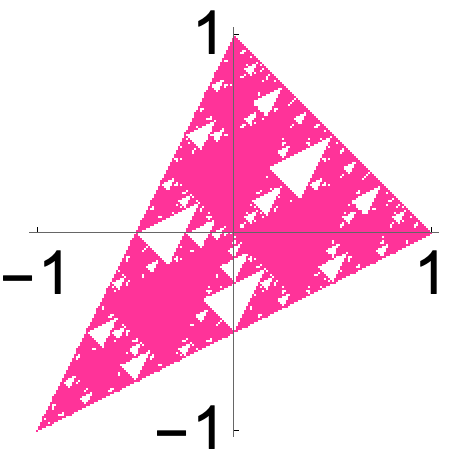} \hskip 2 em
\includegraphics[width=.15\textwidth]{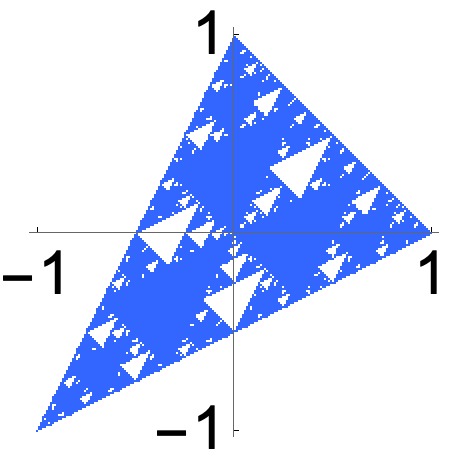}}

For each $\seq \in \seqsp$, define the tilings 
\begin{equation}
\seq_\sq = \bigcup_{\vec{j}\in \Z^m} (\sq + \vec{j}, \seq(\vec{j}))\text{ and }\seq_\fractile = \bigcup_{\vec{j}\in \Z^m} (\fractile + \vec{j}, \seq(\vec{j})),
\label{eq:tilingsfromsequences}
\end{equation}
the latter forming a tiling since $(Q, \digits)$ is a unit digit system.

\begin{claim}
The tiling dynamical systems $(\ssp, \R^m)$ and $(\sat, \R^m)$ are topologically conjugate.
\end{claim}
There are two ways to see that this is true. The first is to note that the spaces are {\em mutually locally derivable}; see \cite[Sec.~5.2]{BG}. To do this, a local map must be constructed that identifies, for each tiling in $\ssp$ and each $x \in \R^m$, exactly what tile to place at $x$ in the corresponding tiling in $\sat$. A local map from $\sat$ to $\ssp$ must also be constructed. In our situation, the local code simply replaces a copy of $\sq$ with a copy of $\fractile$ at the same location (or vice versa), of course keeping the same tile label.
Thus the tiling dynamical systems $(\ssp, \R^m)$ and $(\sat, \R^m)$ are mutually locally derivable, which in particular means they are topologically conjugate.

The second way to prove the claim is to construct a homeomorphism from 
$\ssp$ to $\sat$ directly. The easiest way to do this is to define the homeomorphism first on the {\em transversal}: the set of tilings for which the tile at the origin
is the trivial shift of a prototile. The homeomorphism between the transversals is given by Eq.~\eqref{eq:tilingsfromsequences}. The homeomorphism extends to $\R^m$ since every tiling in either space is the translate of some element of the transversal.

The fact that  $Q\fractile = \bigcup_{\vec{d} \in \digits} \fractile + \vec{d}$ leads to an {\bf \em inflate-and-subdivide rule} for $\sat$:
\[\phi(\fractile, \alet) = \bigcup_{\vec{d} \in \digits} (\fractile + d, \sub(\alet, \vec{d})).\]
For Example~\ref{ex:gasket}, the inflate-and-subdivide rule for the pink fractile is
pictured on the left of Figure~ \ref{fig:FracGasket}, along with its $2$- and $3$-supertiles. This figure should be compared to Figure~\ref{fig:Gasket}. The location of the origin should be noticed in both the square tile and the fractile, because under the qubit substitution the origin of a tile is moved to the elements of $\digits$. The color of each tile is specified by $\sub$.
\begin{figure}[ht]
\centering
\raisebox{.865in}{\includegraphics[width=.125\textwidth]{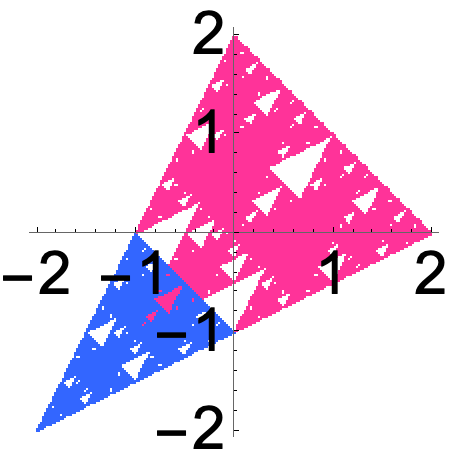}}
 \hskip 2 em
\raisebox{.475in}{\includegraphics[width=.25\textwidth]{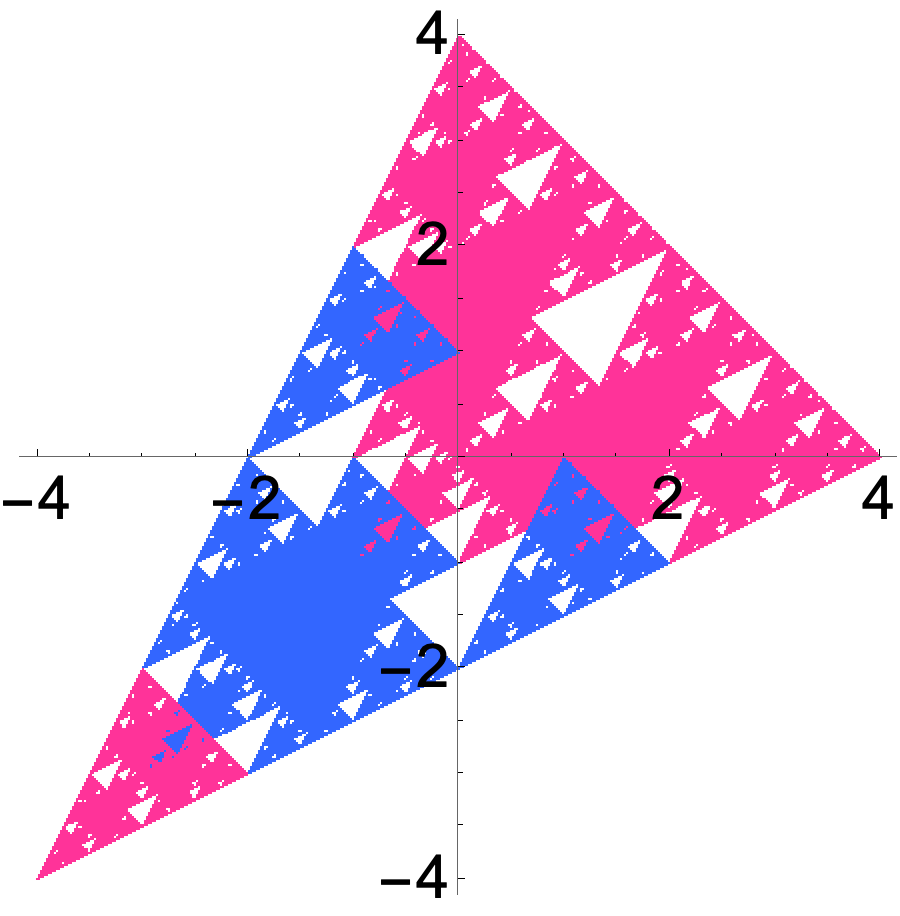}}
 \hskip 2 em
\includegraphics[width=.4\textwidth]{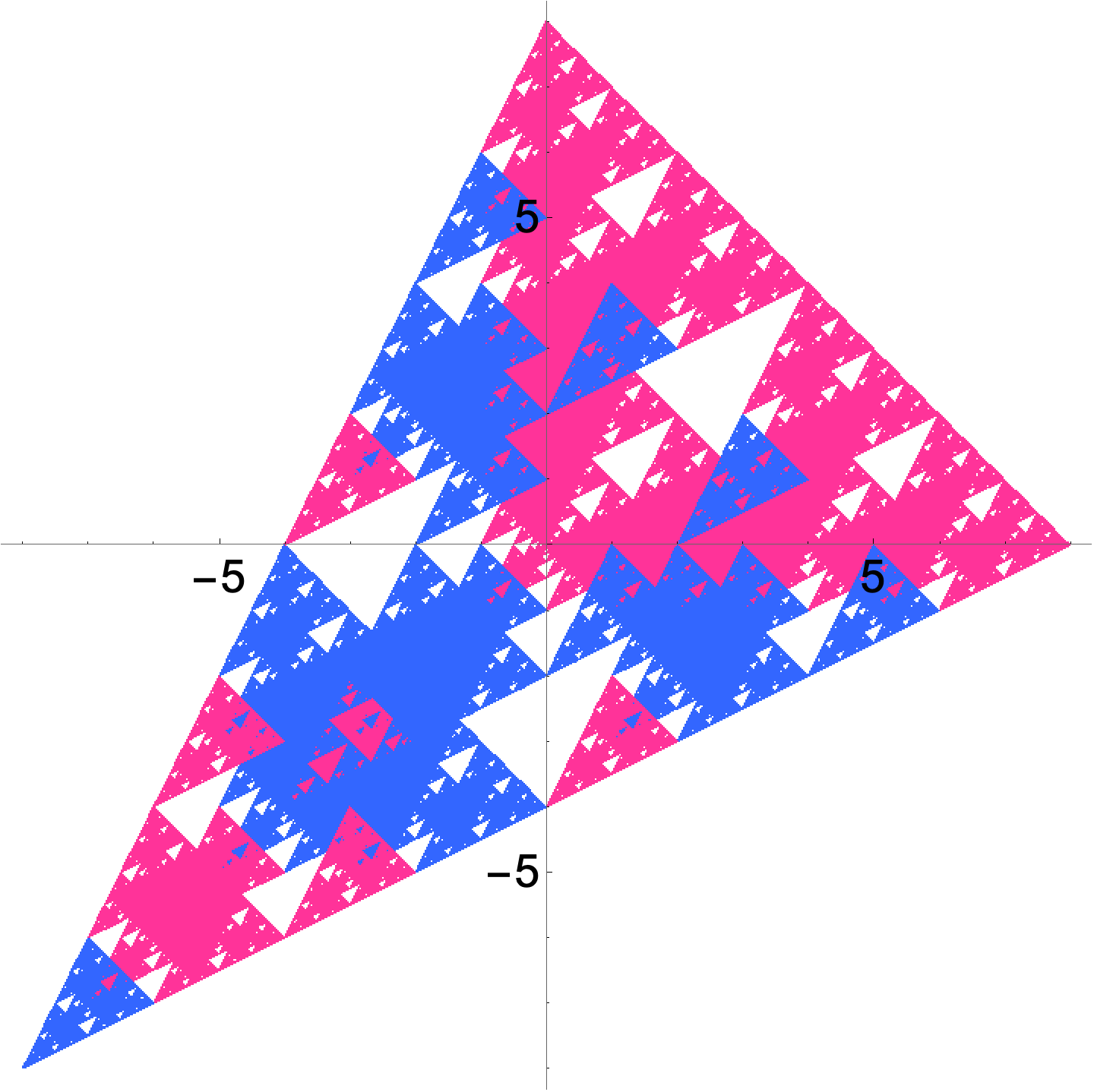}
\caption{Three iterations of the pink $\alet$-tile located at the origin.}\label{fig:FracGasket}
\end{figure} 
 
Figure \ref{fig:gasketpatch} shows a large rectangular word in some $\seq \in \seqsp$ embedded in $\ssp$ and in $\sat$. The image of $\seq$ in either space is constructed by moving the origin of the prototiles to the same location, regardless whether the prototiles are based on $\sq$ or $\fractile$.

 \begin{figure}[ht]
\centering
\includegraphics[width=.65\textwidth]{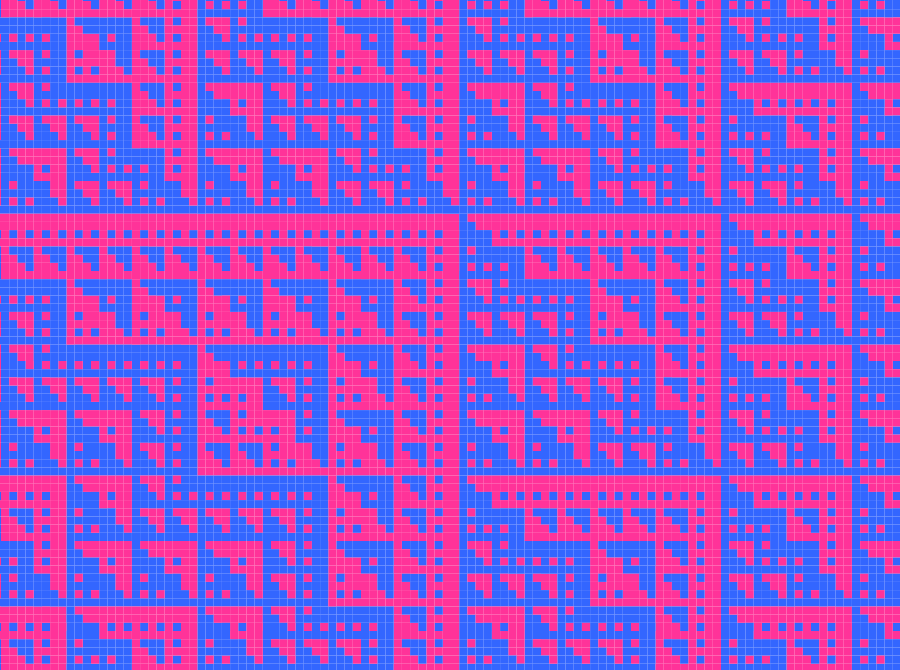}
\vskip 2 em

\includegraphics[width=.65\textwidth]{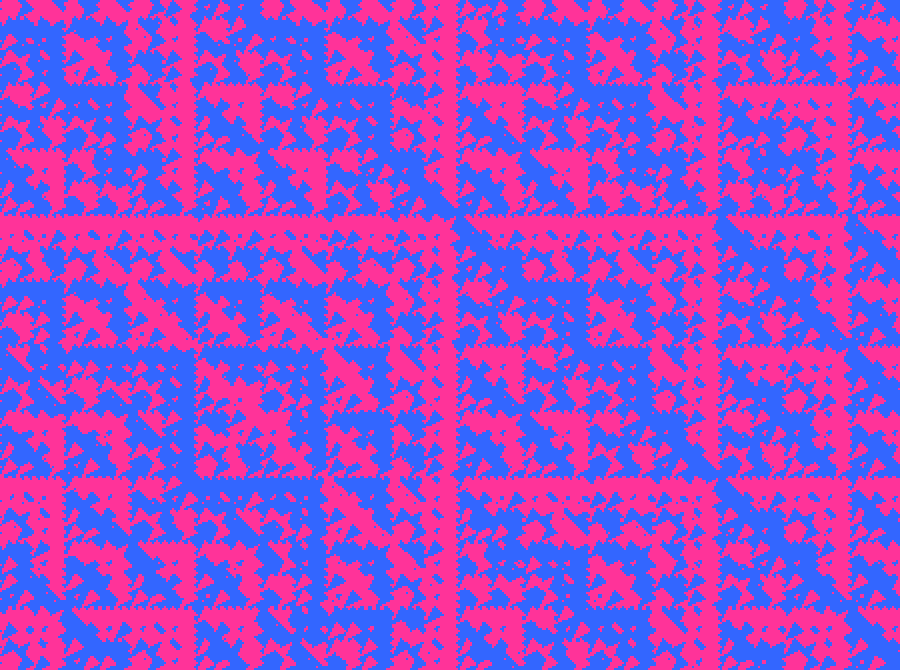}
\caption{The image of some $\seq \in \seqsp$ embedded into the tiling spaces $\ssp$ (top) and $\sat$ (bottom).}
\label{fig:gasketpatch}
\end{figure}

\begin{prop}\label{prop: strict ergod}
Let $\sub$ be a primitive qubit substitution for which $(Q, \digits)$ is a unit digit system. If the substitution subshift $(\seqsp,\mathbb{Z}^m)$ is nonempty, it is strictly ergodic and the unique invariant measure is the frequency measure $\mu$. 
\end{prop}

\begin{proof}
Construct the tiling dynamical systems $(\ssp,\R^m)$ and $(\sat, \R^m)$ for $(\seqsp, \Z^m)$. The fact that $\sat$ is a self-affine tiling dynamical system follows from the presence of the inflate-and-subdivide rule $\phi$: all finite patches in any tiling in $\sat$ appear in an $n$-supertile for some sufficiently large $n$. 
By \cite{Solomyak}, $(\sat,\R^m)$ is strictly ergodic, meaning that $(\ssp, \R^m)$ is as well.  Since $(\ssp, \R^m)$ is a suspension of $(\seqsp, \Z^m)$, we conclude that $(\seqsp, \Z^m)$ is also strictly ergodic.

As the transition matrix is the same for both the substitution $\sub$ and the inflate-and-subdivide rule $\phi$, its right Perron--Frobenius eigenvector contains the frequencies of elements of $\mA$.
\end{proof}

\subsection{Spin substitutions.}
Spin substitutions are qubit substitutions that are completely defined by three things: a digit system $(Q,\digits)$, a finite abelian group $G$ called the {\bf \em spin group}, and a map $W: \digits \times \digits \to G$. 

From this, the alphabet is defined to be $\mA = G \times \digits$ with letters written as formal products $\alet = g\vecd$. When $\dlet = e\vecd$, where $e$ is the identity of $G$, 
we call it {\bf \em spin-free}. 
From here onwards, we follow the convention that $\dlet$ is a (spin-free) element of the alphabet and $\vec{d}$ specifies a spatial location, either in $\Z^m$ or in the corresponding odometer, as we shall see later. By abuse of notation, we also denote by $\digits$ the subalphabet of $\mathcal{A}$ containing all spin-free letters. It will be clear from the context whether $\digits$
is used as a spatial object, or as a subalphabet. 
Note that $G$ acts on $\mA$ via $g(\tilde{g}\dlet)=(g\tilde{g})\dlet$. 

The canonical projections $\spin$ and $\dig$ on $\mA$ are called the {\bf \em spin map}  and {\bf \em digit map}, respectively. These projections can be extended to maps on the subshift. In this case, the spin map commutes with multiplication in $G$ but forgets its location, while the digit map forgets the spin. Letters in the alphabet therefore keep track of both spatial and algebraic information.

\begin{definition}
The {\bf \em spin substitution} $\sub$ given by $(Q, \digits, G, W)$ is defined for spin-free letters $\dlet = e \vecd$ in $\mA$ to be
\[
\sub(e\,\vecd, \vecd') = W(\vecd,\vecd') \, \dlet^{\prime}, \text{ for all } \vecd' \in \digits,
\] 
and extends to the rest of $\mA \times \digits$ via
\[
\sub(g\,\vecd, \vecd') = g \,W(\vecd,\vecd') \, \dlet^{\prime}.
\] 
\end{definition}

 Note that $\sub(\alet,\vecd)$ is indeed an element of $\mA$ since $\spin(\alet)$ and $ W(\dig(\alet),\vecd)$ are both in $G$. In particular, $\sub(\alet,\vecd)$ is an element of the equivalence class of letters \[ [ \dlet ]=\left\{g\dlet\mid g\in G\right\} \subset \mA.\]
 Put another way, no matter which $\alet \in \mA$ is substituted, the letter at $\vecd$ in $\sub(\alet)$ will be in $[\dlet]$. In this way, {\em spin substitutions keep track of their underlying digit structure}.

To summarize, spin substitutions satisfy two key properties. For each $\vecd\in \mathcal{D}$ and $\alet \in \mA$
\begin{enumerate}
\item[\textbf{(R1)}] 
$\mathcal{S}(\alet, \vecd)\in \big[\dlet\big]$ 
\item[\textbf{(R2)}] $\mathcal{S}(g\alet, \vecd)=g \mathcal{S}(\alet, \vecd)$.
\end{enumerate}

\begin{example}[Rudin--Shapiro as spin substitution] The one-dimensional Rudin--Shapiro substitution is the spin substitution given by $(Q, \digits, G, W)$, where $Q = 2$, $\digits = \{0, 1\}$, $G = C_2 = \{1, -1\}$, and $W = \begin{pmatrix}1 & 1 \\ 1 & -1 \end{pmatrix}$.
\exend
\end{example}

\begin{example}[Triomino substitution]\label{ex:triomino}
This is one of the simplest examples that generalizes Rudin--Shapiro type substitutions to three spins and digits while having reasonable geometry.
The digit system is given by $Q=\begin{pmatrix}
2& 1\\
-1& 1
\end{pmatrix}$ with digit set 
\[\digits=\big\{(0,0),(1,0),(0,1)\big\}=\left\{\vecd_1,\vecd_2,\vecd_3\right\}.\] 
Choosing $G=C_3=\left\{1, \omega, \omega^2\right\}$ with $\omega=\exp(2\pi i/3)$, one gets the nine-letter alphabet $\mathcal{A}=\big\{\alet=\omega^{j}\dlet_i,\, 0\leqslant j\leqslant 2,\, 1\leqslant i\leqslant 3\big\}$.

The letters in $\mathcal{A}$ are depicted in Figure~\ref{fig:level-0 TripleRS} as colored unit squares whose position will be given by the lower left corner, which is assumed to be in $\Z^2$. The spins are associated with colors, with spin-free letters appearing in shades of red in the leftmost column. The spin $\omega$ is represented in shades of green, and $\omega^2$ in blue.

\begin{figure}[ht]
\includegraphics[width = 2in]{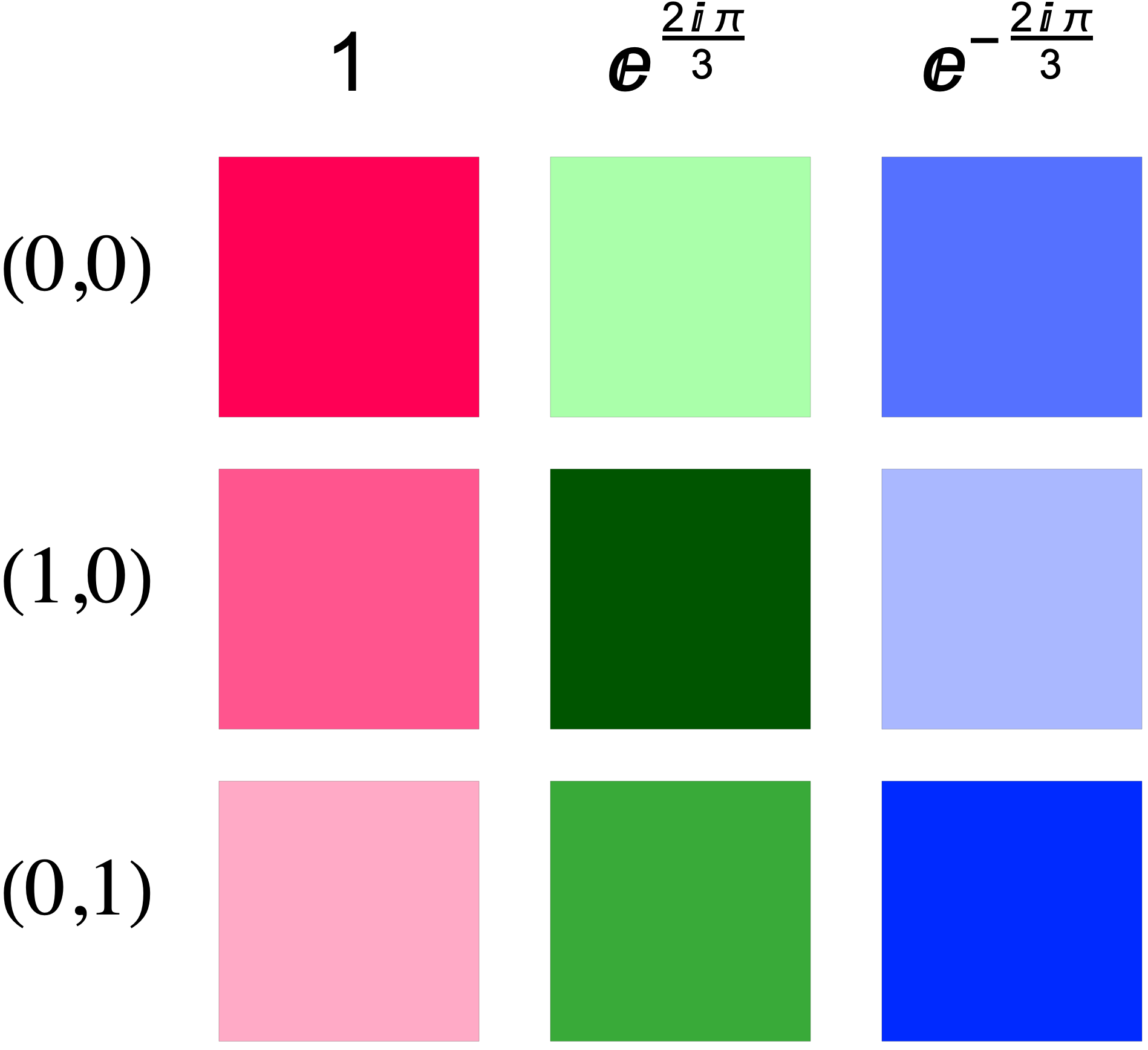}
\caption{The alphabet. Spins are depicted as colors that vary in shade by digit.}\label{fig:level-0 TripleRS}
\end{figure}

\bigskip

The map $W$ is taken to be
$W\colon (\vecd_i,\vecd_j)\mapsto \omega^{(i-1)(j-1)}$, the (Vandermonde) matrix
\[
W=\begin{pmatrix}
1 & 1 & 1 \\ 1 & \omega & \omega^2 \\ 1 & \omega^2 & \omega
\end{pmatrix}.
\]
 The substitution for the spin-free letters are given by the rows of $W$:
\begin{align}
\begin{split}
\stri(\dlet_1,\digits) &= \{\dlet_1, \dlet_2, \dlet_3\}\\
\stri(\dlet_2,\digits) &= \{\dlet_1, \omega \, \dlet_2, \omega^2 \dlet_3\}\\
\stri(\dlet_3,\digits) &= \{\dlet_1, \omega^2 \dlet_2, \omega\, \dlet_3\}
\end{split},
\label{eqn:trisubdef}
\end{align}
from which the rest of the substitution is defined via \RT.

The 1- and 2-supertiles are shown in Figure~\ref{fig: level-2 Triple RS}, arranged in the same way as in Figure~\ref{fig:level-0 TripleRS}. The leftmost column in each grid represents the supertiles for the spin-free digits. This column comes directly from the corresponding equations \eqref{eqn:trisubdef}. The remaining columns come from multiplication by $\omega$ and $\omega^2$.

\begin{figure}[ht]
\raisebox{.5in}{\includegraphics[width = 2in]{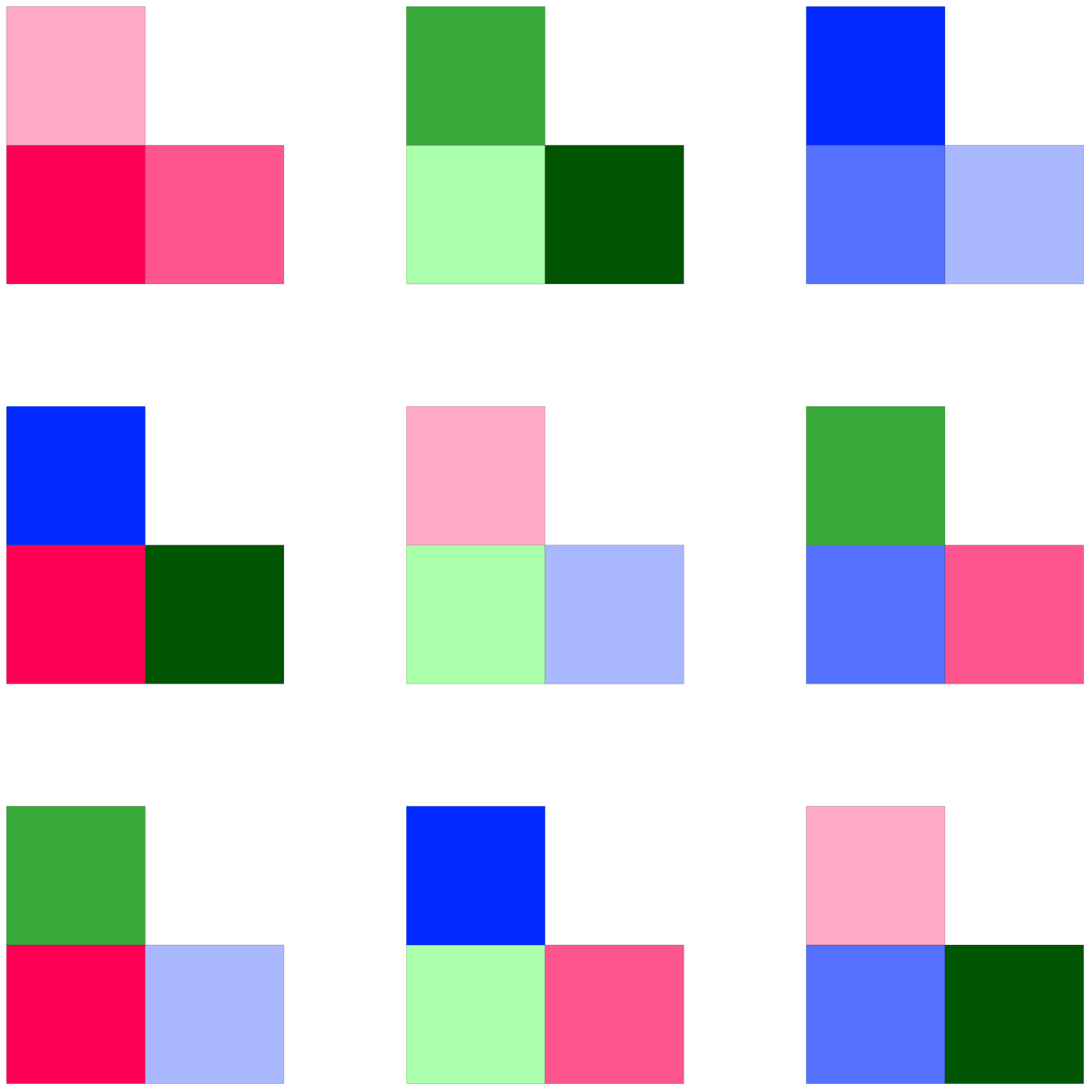}} \hspace{.5in}
\includegraphics[width = 3in]{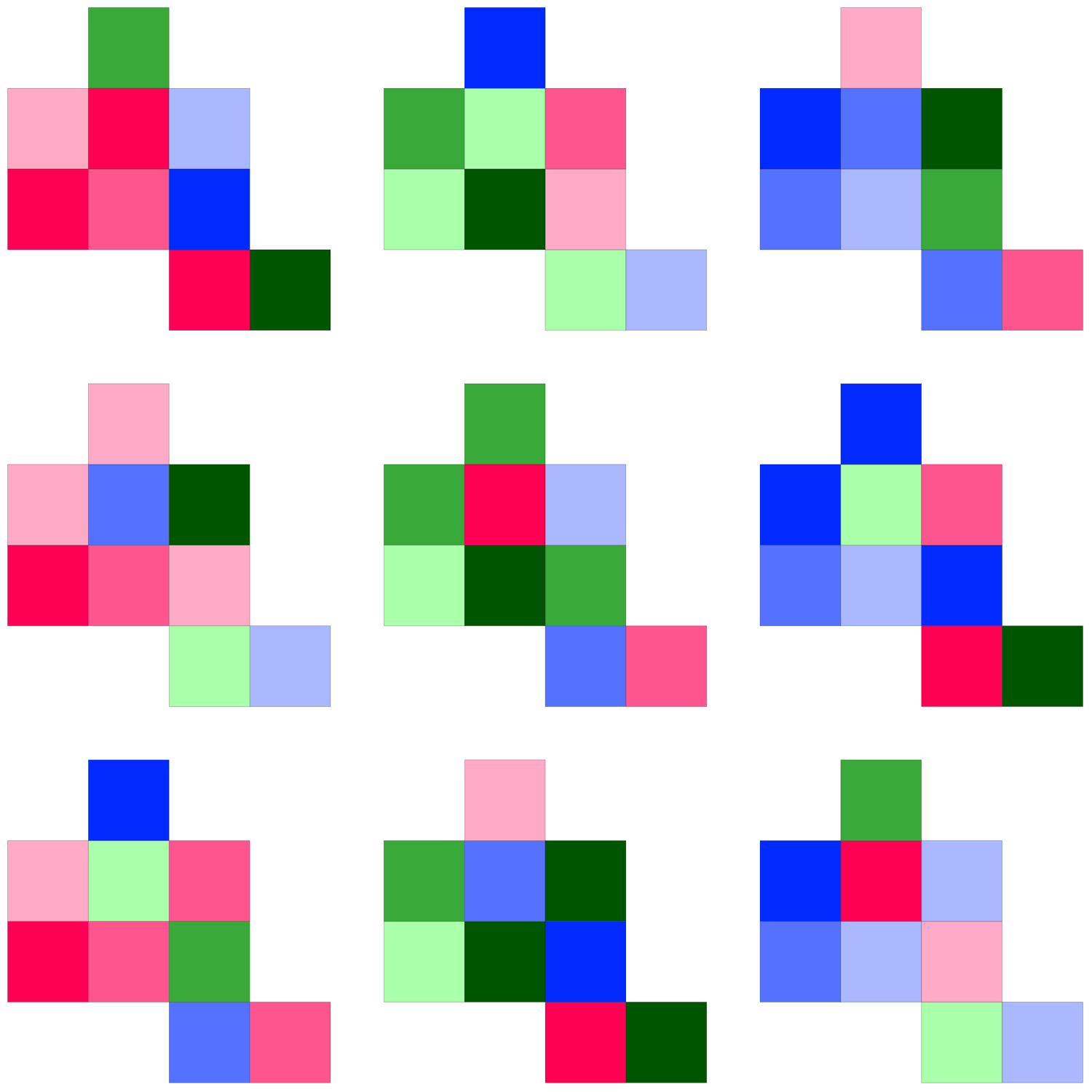}
\caption{The 1- and 2-supertiles for the triomino substitution, arranged by digit and spin as in Figure~\ref{fig:level-0 TripleRS}.}
\label{fig: level-2 Triple RS}
\end{figure}

Figure \ref{fig:tribigtile} shows $\stri^{10}(\dlet_1)$, where $\dlet_1$ is the tile in the upper left of Figure~\ref{fig:level-0 TripleRS}. 
\exend

\begin{figure}
\includegraphics[width=4in]{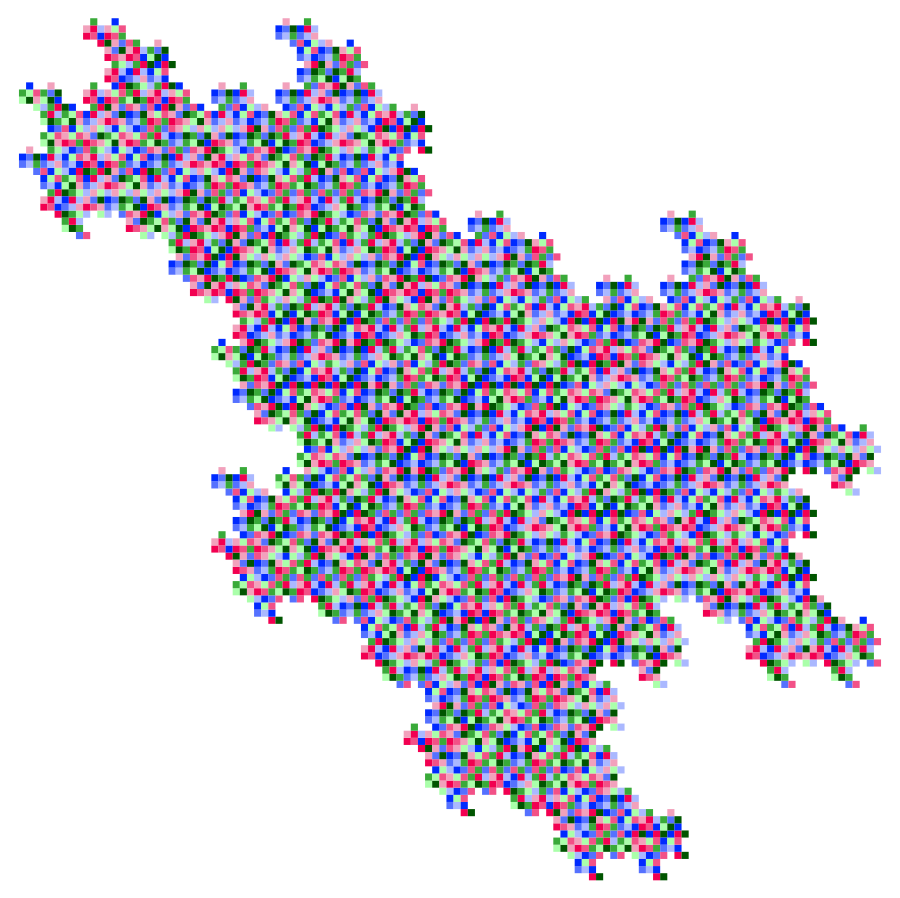}
\caption{An indication that spins are uniformly distributed in the triomino subshift.}
\label{fig:tribigtile}
\end{figure}

\end{example}

\begin{remark}

No successful attempt was made to have the spin colors reflect multiplication in $G$, but still there is a reason to associate the colors with spins rather than digits. The spin map $\spin$ produces a natural factor map that forgets the digits. In this case, a single shade of each color represents tilings in $\spin(\seqsp)$. If instead we choose the transposed color scheme wherein the digits are associated with colors and the shades depend on the spin, this factor map yields periodic tilings regardless of choice of $W$. \exend

\end{remark}

\section{Dynamical spectrum of spin substitutions}\label{sec:dynamspec}

\subsection{Spectral theory on \texorpdfstring{$L^2(\seqsp,\mu)$}{L2}}\label{sec:spectheory}
We briefly recall some notions in the spectral theory of dynamical systems. Let 
$(\seqsp,\mathbb{Z}^m,\mu)$ be a measure-theoretic dynamical system with invariant measure $\mu$. For $\vec{j}\in\mathbb{Z}^m$, we have the unitary operator $U_{\vec{j}}:L^{2}(\seqsp,\mu)\to L^{2}(\seqsp,\mu)$ given by 
$(U_{\vec{j}}f)(\T)=f(\mathcal{T}-\vec{j})$. The {\bf \em spectral or Fourier coefficients} of $f\in L^{2}(\seqsp,\mu)$ are defined for each $\vec{j} \in \Z^m$ to be
\[
\widehat{f}(\vec{j})=\big\langle U_{\vec{j}}f\mid f\big\rangle  =\int_{\seqsp} f(\mathcal{T}-\vec{j})\overline{f(\mathcal{T})}\, \dd\mu(\mathcal{T}).
\]
 Since the function $\widehat{f}: \vec{j}\mapsto  \big\langle U_{\vec j}f\mid f\big\rangle$ is positive definite, by the Herglotz--Bochner theorem, there exists a unique positive measure $\sigma_f\in\mathcal{M}^{+}(\mathbb{T}^d)$ with 
\[
\widehat{f}(\vec{j})=\int_{\mathbb{T}^d}z^{\vec{j}}\,\dd\sigma_{f}(z). 
\]
The measure $\sigma_f$ is called the {\bf \emph{spectral measure}} of $f$.
For $f\in L^{2}(\seqsp,\mu)$, its associated cyclic subspace $Z(f)$ is defined as $Z(f)=\overline{\text{span}\big\{(U_{\vec{j}}f)\mid\vec{j}\in \mathbb{Z}^m\big\}}$, which extends to $Z(\left\langle f_1,\ldots,f_r\right\rangle)$, $\left\{f_1,\ldots,f_r\right\}$ being a finite set. 
The positive measure $\sigma_f$ admits a generalized Lebesgue decomposition 
\[
\sigma_f=(\sigma_f)_{\textsf{pp}}+(\sigma_f)_{\textsf{ac}}+(\sigma_f)_{\textsf{sc}}. 
\]
into its pure point, absolutely continuous, and singularly continuous components, not all of which are trivial. 

An {\bf \em  eigenvalue} of the $\ZZ^{m}$-action of a measure-theoretic dynamical system $(\seqsp,\ZZ^{m},\mu)$ is an $\alpha\in\mathbb{R}^{m}$ for which there exists $f\in L^{2}(\seqsp,\mu)$ with $f\not\equiv 0$ that satisfies $U_{\vec{j}}f=\ee^{2\pi\ii \left\langle \alpha\mid \vec{j}\right\rangle}f$ for all $\vec j \in \Z^m$. An eigenvalue is called {\em continuous} if it has a continuous eigenfunction.

\subsection{Spin group representation} Spin substitutions have an automorphism on $\seqsp$ given by  $(g\T)(\vec{j})=g(\T(\vec{j}))$ that commutes with the shift by $\RO$. This induces an action on $L^2(\seqsp,\mu)$ given by
\begin{equation}\label{eqn:group translation}
(U_g f)(\T) = f(g\T) \text{ for all } \T \in \seqsp \text{ and } g \in G.
\end{equation}

Let $\Gdual$ be the dual group of $G$, i.e. the group of continuous characters $\chi: G \to S^1$.  We say $f \in L^2(\seqsp,\mu)$ is an {\bf \em eigenfunction} for $U_g$ if there exists $\chi \in \Gdual$ such that
\begin{equation}
(U_g f)(\T) = \chi(g) f(\T) \text{ for all } \T \in \seqsp \text{ and } g \in G.
\end{equation}
This helps provide spectral information through its own eigenfunctions, as we shall see below. Every function $f\in L^{2}(\seqsp,\mu)$ admits the spectral decomposition 
\begin{align}\label{eq:decomp function}
f(\T)=\frac{1}{|G|}\sum_{\chi\in \Gdual}\sum_{g\in G} \chi(g)f(g\T)=\sum_{\chi\in \Gdual} f^{\chi}(\T), 
\end{align}
where $f^{\chi}$ is an eigenfunction of $U_g$ with eigenvalue $\chi(g)$. 

For $g\dlet\in \mathcal{A}$, 
define $\one^{ }_{g\dlet}\in L^{2}(\seqsp,\mu)$ via $\one^{ }_{g\dlet}(\mathcal{T})=1$ whenever $\mathcal{T}(0)=g\dlet$ and zero otherwise. 
Recalling that $[\dlet]$ is the equivalence class of elements of $\mathcal{A}$ that share the same digit as $\dlet$, we define the indicator function of $[\dlet]$ as  
\[
\one^{ }_{[\dlet]}=\sum_{g\in G}\one^{ }_{g\dlet} ~.
\]
Given a character $\chi\in \Gdual$, define
\[
f^{\chi}_{\dlet}=\chi(\spin(\mathcal{T}))\one^{ }_{[\dlet]}.
\]
We will see that the spectral analysis on $L^{2}(\seqsp,\mu)$ can be narrowed down to  these functions.
\begin{example}
Let $\mathcal{S}$ be the Triomino substitution of Example~\ref{ex:triomino}.
For a fixed spin-free letter $\dlet\in \mathcal{A}$, we have 
\begin{align*}
f^{\chi_0}_\dlet&=\one^{ }_{\dlet}+\one^{ }_{\omega \dlet}+\one^{ }_{\omega^2 \dlet}\\
f^{\chi_1}_\dlet&=\one^{ }_{\dlet}+\omega \one^{ }_{\omega \dlet}+\omega^2\one^{ }_{\omega^2 \dlet}\\
f^{\chi_2}_\dlet&=\one^{ }_{\dlet}+\omega^2 \one^{ }_{\omega \dlet}+\omega \one^{ }_{\omega^2 \dlet}.
\end{align*}
where $\chi^{ }_i\in \widehat{C_3}$ ~. \exend
\end{example}

\subsection{Statement of results}\label{sec: main results}

We now have our main results regarding the dynamical spectrum of $(\seqsp,\mathbb{Z}^m,\mu)$, to be proved in the sections that follow. These results are generalizations of \cite[Thm.~4.1]{Frank1} where the spin matrices are real Hadamard matrices.

\begin{prop}\label{prop: L2 decomp}
Let $(\seqsp,\mathbb{Z}^m,\mu)$ be the measure-theoretic dynamical system induced by the subshift of a primitive, aperiodic qubit spin substitution $\mathcal{S}$ arising from $(Q,\mathcal{D})$ with map $W$ and spin group $G$. Then, one has 
\[
L^{2}(\seqsp,\mu)=\bigoplus_{\chi\in\Gdual}H^{\chi}, \quad\text{ where }\quad  H^{\chi}=Z\left(\left\langle f_{\dlet}^{\chi} \right\rangle\right). 
\]
\end{prop}

This follows as a corollary of Theorem~\ref{thm:skew product} below, which shows that the subshift of a qubit spin substitution is measure-theoretically isomorphic to a skew product over the corresponding odometer. Note that the function $f^{\chi}$ in Eq.~\eqref{eq:decomp function} is always in $H^{\chi}$, for all $f\in L^{2}(\seqsp,\mu)$.

Proposition \ref{prop: L2 decomp} means we can do a separate spectral analysis on $H^{\chi}$ for each character $\chi\in\Gdual$. In fact, for each $H^{\chi}$, we have the following spectral purity result, which follows from an extension of the analysis by Helson in \cite{Helson} to $\Z^m$-actions; see Theorem~\ref{thm:skew product} and Proposition~\ref{prop:purity-skew}.

\begin{prop}\label{prop:purity-subshift}
The subrepresentation $U_{\vec{v}}|_{H_{\chi}}$ is spectrally pure. Equivalently, for a fixed $\chi\in \widehat{G}$, all functions in $H^{\chi}$ have the same spectral type, which is either pure point, purely absolutely continuous or purely singular continuous. \qed
\end{prop}

The next theorem provides sufficient conditions under which $H^{\chi}$ contains pure point, purely absolutely continuous or purely singular continuous spectral measures and bounds on the corresponding spectral multiplicities. 
We first define the following notions which depend on the spin map $W$. 

\begin{definition}\label{def:unitaryrank1}
Let $\mathcal{S}$ be a qubit substitution with digit set $\mathcal{D}$ and spin matrix $W$, and let $\chi\in \Gdual$. 
We call $\mathcal{S}$ 
\begin{enumerate}
\item {\bf\emph{$\chi$-unitary}} if 
$\frac{1}{\sqrt{|\digits|}}\chi(W)$ is a unitary matrix
\item {\bf \emph{$\chi$-rank-1}} if $\chi(W)$ is a rank-1 matrix.
\end{enumerate}
\end{definition}

\begin{remark}
The condition of being $\chi$-unitary is the generalization of property \textbf{(H)} in \cite{Frank1} and is equivalent to the following: 
For every $\vecd_i\neq \vecd_j\in  \mathcal{D}$, one has 
\begin{equation}
\sum_{\dlet\in\mathcal{D}} \chi(\spin(\mathcal{S}_{\vecd_i}(\dlet)))\,\overline{\chi(\spin(\mathcal{S}_{\vecd_j}(\dlet)}))=0.
\end{equation}
Being $\chi$-rank-1 implies that the factor substitution is bijective; see Section~\ref{sec:factors-sc}. \exend
\end{remark}

As we shall see next, the two conditions in Definition~\ref{def:unitaryrank1} have explicit consequences for the spectral types present in the corresponding subspace $H^{\chi}$. We emphasize that these spectral results neither depend on the geometry of supertiles, nor on the dimension of the space $\ZZ^m$ where the elements of $\seqsp$ live.

\begin{theorem}\label{thm:main result}
Let $\mathcal{S}$ be an aperiodic primitive qubit spin substitution in $\ZZ^m$ with digit set $\mathcal{D}$ and spin matrix $W$. 
Fix $\chi\in \Gdual$.
\begin{itemize}
\item[(a)] If $\chi=\chi_0$ is the trivial character, every function $f\in H^{\chi_0}$ has pure point spectral measure. The corresponding group of eigenvalues is
\[
E_{\mathcal{O}}=\bigcup_{i\geqslant 0} (Q^{T})^{-i}\ZZ^{m},
\]
which are the eigenvalues of the corresponding odometer.
\item[(b)] If $\sub$ is $\chi$-unitary, the associated subspace $H^{\chi}$
decomposes as \[
H^{\chi}=\bigoplus_{\dlet\in\mathcal{D}} Z(f^{\chi}_\dlet)
\]
and the functions $\big\{f^{\chi}_{\dlet}\big\}$, $\dlet\in \mathcal{D}$ all have $\frac{1}{|\digits|}\mu^{ }_{\textnormal{L}}$ as their spectral measure, where $\mu^{ }_{\textnormal{L}}$ is Lebesgue measure on $\mathbb{T}^m$. In this case, the spectral multiplicity of $H^{\chi}$ is exactly $|\digits|$.
\item[(c)] If $\sub$ is $\chi$-rank-$1$, $\chi$ induces a factor map to a bijective substitution $\sub^{(\chi)}$ on the alphabet $G/\textnormal{ker}(\chi)$. 
\begin{itemize}
\item[(i)] Every $f\in H^{\chi}$ has singular spectral measure. Moreover, the spectral multiplicity of $H^{\chi}$ is bounded from above by $|G/\textnormal{ker}(\chi)|$.
\item[(ii)]  If further $\sub^{(\chi)}$ is aperiodic, $H^{\chi}$ has purely singular continuous spectrum that is generated by generalized Riesz products.
\end{itemize}
 
\end{itemize}
\end{theorem}

We split the proof of this main spectral result in several sections. We prove part~(a) regarding the odometer factor in Section~\ref{sec:odometer}. Part~(b), which takes advantage of the cocycle structure satisfied by the spin, will be proved in Section~\ref{sec:ac}. 

For the proof of part~(c), which includes the presence of singular continuous components, we will need the analysis via substitutive factors and their diffraction measures in Section~\ref{sec:diffraction}. The complete proof of part~(c) is then provided in Section~\ref{sec:factors-sc}.

\begin{remark}
The spectral classification result in Theorem~\ref{thm:main result} extends to the case when $G$ is an arbitrary compact Hausdorff abelian group, but there one would need to assume that the qubit substitution is  recognizable and uniquely ergodic.
In this setting, we get that the corresponding subshift is of infinite local complexity \cite{FS2}, and then it is possible that aperiodicity no longer guarantees recognizability.
An example in one dimension is given by $\digits=\left\{0,1\right\}$, $G=S^1$ and where the spin matrix is
\[
W=\begin{pmatrix}
1 & -1\\
\alpha & \alpha
\end{pmatrix},
\] 
where $\alpha\in S^1$ is an irrational rotation. Here, $\Gdual\simeq \mathbb{Z}$. This example has countably infinite Lebesgue component,  countably infinite singular continuous component and a pure point component corresponding to the dyadic odometer. We refer the reader to \cite{MRW} for details regarding this example. \exend
\end{remark}

\subsection{Digit odometers, skew product representation, and pure point spectrum}\label{sec:odometer}

The general framework for odometers in $\ZZ^m$  was introduced in \cite{Cortez}, which generalizes both the notion of an $L$-adic odometer in one dimension and also products of $L_i$-adic odometers for $i \in \left\{1,\ldots,m\right\}$.
We will show that a qubit substitution with expansion matrix $Q$ has an underlying $Q$-odometer per \cite{Cortez}. Such $Q$-adic odometers have been used by
 \cite{LMS-2} in the context of substitution Delone sets on lattices and by \cite{Vince} in the context of digit tilings.
We recall some notions below and discuss the explicit formulation in the context of qubit substitutions.

Let $\ZZ^m= Z_1 \supseteq Z_2 \supseteq \cdots$ be a sequence of subgroups isomorphic to $\mathbb{Z}^m$. 
Since $Z_i \supseteq Z_{i+1}$, we have that $ \ZZ^{m}/Z_{i} \subseteq  \ZZ^{m}/Z_{i+1}$ and this inclusion gives rise to the inverse limit 
\[
\odom=\lim_{\leftarrow i} \ZZ^{m}/Z_i \subset \prod_{i\geqslant 1} \ZZ^{m}/Z_i.
\]
An element $\addr{k} =\big( \odigit{k}_i \big)_{i \ge 1} \in \odom$ satisfies $\odigit{k}_{i}\equiv \odigit{k}_{i+1} \pmod{Z_i}$. 
This set forms a group under coordinate-wise modular addition, defined for $\addr{k}, \addr{j} \in \odom$ to be
\[ \addr k \oplus \addr j = \left( \odigit{k}_i + \odigit{j}_i \pmod{Z_i}
\right)_{i \ge 1}
\]

The action of $\Z^m$ on $\odom$ is given by extending $\oplus$ to elements $\vec{j}  \in \Z^m$. For $\addr{k} \in \odom$, define
\[\addr k \oplus \vec j = \left( \odigit{k}_i + \vec{j} \pmod{Z_i}
\right)_{i \ge 1}.
\] This can be thought of as embedding $\Z^m$ into $\odom$ and then adding using the existing group addition. An essential fact is that $\ZZ^m$ acts freely on $\odom$ if and only if $\cap_{i\geqslant 0} Z_i=\left\{0\right\}$. 
The dynamical system $(\odom,\ZZ^m)$ is called a {\bf \em $\ZZ^m$-odometer}. 

Let $\sub$ be a qubit substitution with expansive map $Q$ and digit set $\digits$. For each $i$, let $Z_i=Q^{i}\ZZ^{m}$ and construct the odometer $\odom = \lim_{\leftarrow i} \ZZ^{m}/Q^{i}\ZZ^m$. Addition $\oplus$ is well-defined regardless of any specific choice for the equivalence classes modulo $Q^{i}\ZZ^m$. Since $Q$ is expansive, we have that $\cap_{i\geqslant 0} Q^i\ZZ^m=\left\{0\right\}$, so we know addition acts freely on $\odom$. 

Since $\digit{i}$ is a complete residue system for $Q^{i}\ZZ^m$, it provides a convenient set of coset representatives. We thus identify
\[ \odom= \lim_{\leftarrow i} \ZZ^{m}/Q^{i}\ZZ^m \cong \lim_{\leftarrow i} \digit{i}.
\]

We endow $\odom$ with the product topology and choose a measure $\nu$ to be the {\bf \em frequency measure of $\odom$} as follows. For each $M \in \N$ and $\vecd \in \digit{M}$ define a {\bf \em cylinder set} 
\[ 
\mathcal{Z}^{(M)}_{\vec{d}}= \left\{\left( \odigit{k}_i \right)_{i \ge 1} \in \odom \text{ such that } \odigit{k}_M \equiv \vec{d} \pmod{\digit{M}}
\right\}.
\]
The frequency measure $\nu$ is defined on all such $M$ and $d$ to be $\nu\big(\mathcal{Z}^{(M)}_{\vec{d}}\big)=\frac{1}{|\digits|^{M}}$, which extends to Borel subsets of $\odom$ since cylinder sets form a basis of the product topology.

\begin{definition}\label{def:subsodometer}
The {\bf \em digit odometer} associated to a qubit substitution with expansion matrix $Q$ is the odometer $\odom = \lim_{\leftarrow i} \digit{i}$ endowed with the product topology and digit frequency measure $\nu$.
\end{definition}

\begin{prop}\label{prop:odometer factor map}
The subshift of an aperiodic qubit substitution factors measure-theoretically onto its digit odometer. 
\end{prop}

\begin{proof}
Aperiodicity implies that we have the unique composition property \cite{Solomyak2}, which means that for $\mu$-a.e. every $\mathcal{T}\in \seqsp$ and every $M\in\mathbb{N}$, there exist a unique $\mathcal{T}^{(M)}\in \seqsp$ and $\vec{j}_{M}\in \digits^{(M)}$ such that 
\begin{equation}\label{eq:unique decomp}
\mathcal{T}=\sub^{M}(\mathcal{T}^{(M)})-\vec{j}_{M}.
\end{equation} 

We write $\addr{j}:=\big(\vec{j}_M\big)_{M\geqslant 1}$.
Let $\seqsp^{\prime}\subset \seqsp$ be the subset of full $\mu$-measure for which every element has a unique decomposition above for each $M$. 
We now define the coding $\Theta$ of $\seqsp^{\prime}$ onto the odometer as 
\begin{equation}\label{eq:odom coding}
\Theta\colon \seqsp^{\prime}\to \odom,\quad \quad \Theta(\T)= \addr{j},
\end{equation} 
$\addr{j}$ being the unique sequence of translation vectors 
in Eq.~\eqref{eq:unique decomp}, where every $i$th component is seen as an element of $\ZZ^{m}/Q^{i}\ZZ^{m}$. 
This map is surjective since every sequence in $\odom$ gives rise to an element $\mathcal{T}$ in $\seqsp$. Now let $\vec{v}\in\ZZ^m$ be fixed. We show that 
$\vec{v} \circ \Theta=\Theta\circ \vec{v}$.

Fix $\mathcal{T}\in\seqsp^{\prime}$ whose odometer coding is $\addr{j}:=\big(\vec{j}_M\big)_{M\geqslant 1}$. For a fixed $M$, there exists $\T^{(M)\prime}\in \seqsp$ such that $\mathcal{T}-\vec{v}=\sub^{M}(\mathcal{T}^{(M)\prime})-(\vec{j}_M+\vec{v})\mod Q^{M}\Z^m$, where $\T^{(M)\prime}$ is an appropriate translate of $\T^{(M)}$ in Eq.~\eqref{eq:odom coding}. We then get $\big(\Theta(\mathcal{T}-\vec{v})\big)_{M}=(\vec{j}_M+\vec{v})\mod Q^{M}\ZZ^{m}=\big(\vec{v}\circ \Theta(\mathcal{T})\big)_M$, since the projection map commutes with addition for every $M$. 
\end{proof}

We show in the next result that a qubit spin substitution is measure-theoretically isomorphic to a group extension of the odometer $\mathcal{O}$ if it is aperiodic and primitive. First, we recall the definition of a skew product and construct one which fits the setting of spin substitutions. 

\begin{definition}
Let $(X,\mathcal{G},\nu)$ be a measure-theoretic dynamical system with a group action  induced by $\mathcal{G}$. Let $G$ be a compact abelian group equipped with Haar measure $\haar$.  Consider the product space $(X\times G,\nu\times \haar)$ and let $(x,g)\in X\times G$. 
From a measurable map $\phi\colon  \mathcal{G}\times X\to G$ and $v\in\mathcal{G}$, one can build $\Psi_{v}\colon X\times G\to X\times G$ via
\[
\Psi_{v}(x,g)=(v(x),\phi(v,x)g).
\]
This map $\Psi$ is called  a {\bf \em skew product} and the action $\phi$ on the second coordinate is called a {\bf \em cocycle} . The tuple $(X\times G, \mathcal{G},\nu\times \haar)$ is called a {\bf \em compact abelian group extension}, where the action of $\mathcal{G}$ on $X\times G$ is induced by $\Psi$. This is a special type of a skew product dynamical system.
\end{definition}

Let $\sub=(Q,\digits,G,W)$ be an aperiodic primitive qubit spin substitution in $\ZZ^m$ with subshift $\seqsp$ and underlying odometer $\odom$.  Write $\addr{j}=\sum_{n=0}^{\infty}Q^{n}(j^{(n)})$ and $\addr{j}+\vec{v}=\sum_{n=0}^{\infty}Q^{n}(j+v)^{(n)}$. 
Define $\phi\colon \ZZ^m\times\odom\to G$
to be 
\begin{equation}\label{eq:spin cocycle}
\phi(\vec{v},\addr{j})=\left(\prod_{\ell=0}^{M-1}W(j^{(\ell+1)},j^{(\ell)})^{-1}\right)
\left(\prod_{\ell=0}^{M-1}W((j+v)^{(\ell+1)},(j+v)^{(\ell)})\right)
\end{equation}
where
\begin{equation}\label{eq:definition M}
M:=\max \left\{n\in\mathbb{N}\colon j^{(n)}=(j+v)^{(n)} \text{ and }
 j^{(n-1)}\neq (j+ v)^{(n-1)}
\right\}.
\end{equation}
This map defines the cocycle $\Psi_{\vec{v}}$ via
\begin{equation}\label{eq:cocycle Psi}
\Psi_{\vec{v}}(\addr{j},g)=(\addr{j}+\vec{v},\phi(\vec{v},\addr{j})g).
\end{equation} 
One can check that this satisfies the cocycle property $\Psi_{\vec{v}+\vec{w}}=\Psi_{\vec{v}}\circ\Psi_{\vec{w}}$. 

\begin{theorem}\label{thm:skew product}
The dynamical system $(\seqsp,\ZZ^m,\mu)$ is measure-theoretically isomorphic to the group extension $(\mathcal{O}\times G, \ZZ^m,\nu\times \haar)$, where the action of $\ZZ^{m}$ is induced by the cocycle $\Psi$  whose cocycle action $\phi$ on the second coordinate is defined in Eq.~\eqref{eq:cocycle Psi}. 
\end{theorem}

The proof proceeds in a similar manner as the proof for bijective substitutions, which are also skew products over odometers; see also \cite{AL} for a similar construction for chained sequences. In the bijective setting, almost every $\T\in\seqsp$ is completely determined by its odometer coding and its letter $\T(0)$ at the origin; see \cite[Sec.~2.3]{Frank3}. For spin substitutions, we show that, for a subset of $\seqsp$ of full $\mu$-measure, every $\T$ is completely determined by its coding $\addr{j}\in\mathcal{O}$ and its spin $\spin(\T(0))$ at the origin. Before continuing with the proof, we need the following lemma, which 
describes the cocycle structure for the spins of level-$M$ supertiles $\sub^{M}(\dlet)$, where $\dlet$ is a spin-free letter.

\begin{lem}\label{lem: spin cocyle}
For any $\vec{i}\in\mathcal{D}^{(M)}$, one has
\begin{equation}\label{eq: spin level-M}
\spin(\mathcal{S}^{M}_{\vec{i}}(\dlet))=W\big(\dlet,i^{(M-1)}\big)W\big(i^{(M-1)},i^{(M-2)}\big)\cdots W\big(i^{(1)},i^{(0)}\big). 
\end{equation}
\end{lem}

\begin{proof}
We do this by induction. For $M=1$, this is directly satisfied from the definition of $W$. We now assume it is true for $M$ and prove that it then holds 
for $M+1$. Let $\vec{i}\in \mathcal{D}^{(M+1)}$ and $\vec{i}=(i^{(M)},\ldots,i^{(1)},i^{(0)})$ be the $(Q,\digits)$-expansion of $\vec{i}$. One has
$\mathcal{S}^{M+1}(\dlet)=\mathcal{S}^{M}(\mathcal{S}(\dlet))$. Since $\mathcal{S}$ commutes with $\spin$, one then has \[\spin(\mathcal{S}^{M+1}_{\vec{i}}(\dlet))=\spin(\mathcal{S}^{M}_{\vec{i}^{\prime}}(\mathcal{S}_{i^{(M)}}(\dlet))= W(\dlet,i^{(M)})\,\spin(\mathcal{S}^{M}_{i^{\prime}}(i^{(M)})),
\] with $\vec{i}^{\prime}=(i^{(M-1)},\ldots,i^{(0)})$.  The claim then follows by induction. 
\end{proof}

\begin{proof}[Proof of Theorem~\textnormal{\ref{thm:skew product}}]
Let $\T\in \seqsp$ with unique odometer coding $\addr{j}$ and consider a translation vector $\vec{v}\in \ZZ^m$. The odometer coding of $\T-\vec{v}$ is given by $\addr{j}+\vec{v}$, which accounts for the first coordinate of $\Psi_{\vec{v}}(\addr{j},g)$ in Eq.~\eqref{eq:cocycle Psi}. 
 To find the image of the second coordinate, we show that $\spin\big((\T-\vec{v})(0)\big)$ can be computed given $\addr{j}$ and $\spin(\T(0))=g$ via Eq.~\eqref{eq:spin cocycle}, for which we exploit the previous lemma.

Recall from Eq.~\eqref{eq:definition M} that $M$ is the smallest integer such that $\vec{j}_{M-1},(\vec{j}+\vec{v})_{M-1}\in \digits^{(M)}$.
 For $\T$, this means $\T(0),(\T-\vec{j}_{M-1})(0)$ and $(\T-(\vec{j}+\vec{v})_{M-1})(0)$ belong to the same level-$M$ supertile $\sub^{M}(\alet)$ for some $\alet=h\dlet \in \mathcal{A}$.  
 For $M\in\mathbb{N}$, $\vec{j}\in\digits^{(M)}$ and a spin-free letter $\dlet$, $
\spin(\sub^{M}_{\vec{j}}(\dlet))$
is given by Eq.~\eqref{eq: spin level-M} from Lemma~\ref{lem: spin cocyle}. 

 We know that $g=\spin(\T(0))=\spin(\sub^{M}_{\vec{j}_{M-1}}(h\dlet))$ for some $h\in G$. By Property~\RT ,  $h=g\cdot \spin(\sub^{M}_{\vec{j}_{M-1}}(\dlet))^{-1}$, where the second term is completely determined by $W$ and the $(Q,\digits)$-expansion of $\vec{j}_{M-1}$. 

The spin-free type $\dlet$ of the level-$M$ supertile at the origin is given  by $j^{(M)}$. 
 With this, we  have 
\begin{align*}
\spin(\T-\vec{v})(0)&=\spin(\sub^{M}_{(\vec{j}+\vec{v})_{M-1}}(hj^{(M)}))=h\cdot\spin(\sub^{M}_{(\vec{j}+\vec{v})_{M-1}}(j^{(M)}))\\&=g\cdot \spin(\sub^{M}_{\vec{j}_{M-1}}(j^{(M)}))^{-1}\cdot\spin(\sub^{M}_{(\vec{j}+\vec{v})_{M-1}}(j^{(M)})
\end{align*}
This means that the spin at zero and the odometer coding completely determines $\T(\vec{v})$ for all $\vec{v}\in\ZZ^{m}$, hence $\T$ itself. Combining with Eq.~\eqref{eq: spin level-M}, we get that $\spin(\T-\vec{v})(0)=\phi(\vec{v},\addr{j})g$.

 The metric isomorphism $\Gamma\colon \seqsp \to \mathcal{O}\times G$ is then given by
\begin{equation}\label{eq:metric isomorphism}
\Gamma(\T)=(\Theta(\T),\spin(\T(0))), 
\end{equation}
 which satisfies
$\Gamma\circ U_{\vec{v}}= \Psi_{\vec{v}}\circ \Gamma$. 
 It remains to show that $\Gamma$ is measure-preserving, i.e., $\mu(\Gamma^{-1}(V))=(\nu\times \mu_{\text{H}})(V)$ for $V\subset \mathcal{O}\times G$ measurable. It suffices to show this for the cylinder set $V=\mathcal{Z}^{(M)}_{\vecd}\times \left\{g\right\}$, where $\vecd\in \digits^{(M)}$ for some $M$ and $g\in G$. To this end, we have $(\nu\times \mu_{\text{H}})(V)=\frac{1}{|\digits|^{M}}\cdot\frac{1}{|G|}$. The set $\Gamma^{-1}(V)\subset \seqsp$ is the set of all $\T$ whose first $M$ terms of the odometer coding coincides with that of $\vec{d}$ and such that the spin of $\T$ at the origin is $g$.

Let $\vecd_{M}=(d^{(M-1)},\ldots,d^{(0)})$ and with $\vecd_{M}\in\digits^{(M)}$.  
For $0\leqslant \ell \leqslant M-1$, the cylinder set of all tilings with a level-$\ell$ supertile of spin-free tile type $\dlet^{(\ell)}$ has measure $\frac{1}{|\digits|}$.  Multiplying all these contributions for all $0\leqslant \ell \leqslant M-1$ yields that the cylinder set of tilings with the specified supertile structure (up to level $M$) has measure $\frac{1}{|\digits|^{M}}$. Finally, since there are $|G|$ choices for the spin at the origin, we get $\mu(\Gamma^{-1}(V))=\frac{1}{|\digits|^{M}}\cdot\frac{1}{|G|}$, as desired.
\end{proof}

The isomorphism $\Gamma$ in Eq.~\eqref{eq:metric isomorphism} induces a unitary isomorphism of the corresponding function spaces. It then suffices to consider $L^{2}(\mathcal{O}\times G, \nu\times\haar)$. 
The following result is well known for compact abelian group extensions  with discrete group actions
\cite[Thm.~4.2]{Frank3}, \cite[Thm.~2]{Robbie}, \cite{Nadkarni} and \cite[Sec.~3]{EL}.

\begin{prop}\label{prop: splitting peter weyl} One has
$\widetilde{H}=L^{2}(\mathcal{O}\times G,\nu\times \haar)=\bigoplus \widetilde{H}^{\chi}$
with \[\widetilde{H}^{\chi}=\left\{f\mid f(\addr{j},g)=\chi(g)\widetilde{f}(\addr{j}), \widetilde{f}\in L^{2}(\mathcal{O},\nu)\right\}=L^{2}(\odom,\nu)\otimes \chi.\]
Moreover, $\widetilde{H}^{\chi}$ is the eigenspace of $U_g$ to the eigenvalue $\chi(g)$. Each $\widetilde{H}^{\chi}$ is $\Psi$-invariant and is generated by functions of the form $\left\{\chi(\cdot)\one_{[\dlet]}\right\}$, where $\dlet\in \digits$.
\end{prop}

\begin{proof}
The splitting follows from a direct application  of the  Peter--Weyl theorem. The invariance of $\widetilde{H}^{\chi}$ can be shown by a direct computation.
For the last claim, note that the subspace $H^{\chi_0}$ where $\chi_0$ is the trivial character, is composed of functions which only depend on the odometer component. 
The span of the indicator functions $\left\{\one_{[\dlet]}\right\}$ of cylinder sets is dense in $L^{2}(\mathcal{O},\nu)$ and hence $H^{\chi_0}=L^{2}(\mathcal{O},\nu)=Z\big(\left\langle \one_{[\dlet]}\right\rangle\big)$. The claim for other characters follows immediately since $\widetilde{H}^{\chi}= L^{2}(\mathcal{O},\nu)\otimes \chi$.
\end{proof}

\begin{proof}[Proof of Proposition~\textnormal{\ref{prop: L2 decomp}}]
This now follows directly as a corollary of Theorem~\ref{thm:skew product} and Proposition~\ref{prop: splitting peter weyl}. 
\end{proof}

We now have the following spectral purity result. We omit the proof here and refer the reader to \cite{Helson, Queffelec2} for the version of the proof for $\mathbb{Z}$-actions; see also \cite{EL}. 

\begin{prop}\label{prop:purity-skew}
The subrepresentation $U_{\Psi_{\vec{v}}}|_{\widetilde{H}^{\chi}}$ 
is spectrally pure. Equivalently, for a fixed $\chi\in \widehat{G}$, all functions in $\widetilde{H}^{\chi}$ have the same spectral type, which is either pure point, purely absolutely continuous or purely singular continuous. \qed
\end{prop}

Proposition~\ref{prop:purity-subshift} then follows immediately from the previous result.

\begin{remark}
Proposition~\ref{prop:purity-skew} has a very strong implication in the determination of the spectral type of $H^{\chi}$. Since the type is constant in $H^{\chi}$, it is enough to find \emph{one} function $f\in H^{\chi}$ of known spectral type in order to determine the spectral type of the \emph{entire} subspace $H^{\chi}$. This makes the use of diffraction-based techniques possible; see Section~\ref{sec:diffraction}. \exend
\end{remark}

We now focus on the subspace $H^{\chi_0}$ associated to the trivial character. By the previous result, one can identify this with 
$L^{2}(\mathcal{O},\nu)$. Since the odometer is equicontinuous, it has pure point dynamical spectrum with respect to the translation action. The following result gives the structure of the group of eigenvalues. 

\begin{prop}[{\cite[Prop.~1]{Cortez}}]\label{prop:odometer eigenvalues}
Let $\mathcal{O}=\lim_{\leftarrow i} \big(\ZZ^{m}/Z_i,\pi_i\big)$ be a $\ZZ^{m}$-odometer. The set of eigenvalues of $\mathcal{O}$ is given by 
\[
E_{\mathcal{O}}:=\bigcup_{i\geqslant 0}\left\{\alpha\in\mathbb{R}^{m}\mid \left\langle \alpha\mid z\right\rangle \in \Z \text{ for all }z\in Z_i\right\}\subset \mathbb{Q}^{m}.
\]
Moreover, every eigenvalue of $\mathcal{O}$ is continuous. \qed
\end{prop}

\begin{proof}[Proof of Theorem~\textnormal{\ref{thm:main result}(a)}]

From Proposition~\ref{prop:odometer factor map}, every eigenfunction $\widetilde{f}\in L^{2}(\mathcal{O},\nu)$ gives rise to an eigenfunction of $U_{\vec{v}}$ in $L^{2}(\seqsp,\mu)$ via $f=\widetilde{f}\circ \Theta$. Since $L^{2}(\mathcal{O},\nu)$ has pure point spectrum, it has a complete system of eigenfunctions, and hence so does $H^{\chi_0}$. 
The claim regarding the set of eigenvalues follows from Proposition~\ref{prop:odometer eigenvalues} and the fact that, for all $i$, $Q^{i}\ZZ^{m}$ is a full-rank sublattice with dual lattice $(Q^{T})^{-i}\ZZ^{m}$.
\end{proof}

\subsection{Unitarity and absolutely continuous spectrum}\label{sec:ac}
In this section, we provide a sufficient condition for the spectrum of $H^{\chi}$ to be purely absolutely continuous, where the multiplicity function is also completely determined.
We prove an analogue of $\chi$-unitarity for level-$M$ supertiles. The proof uses the fact that the spin labels for the level-$M$ supertiles $\mathcal{S}^M(\alet)$ can all be directly derived from the map $W$ as a cocycle, which is given in Lemma~\ref{lem: spin cocyle}. This recursive structure enables one to transfer level-$1$ $\chi$-unitarity to any finite level.

\begin{prop}\label{prop: level-M property V}
Let $\mathcal{S}$ be an aperiodic primitive qubit spin substitution with spin matrix $W$ and let $\chi\in \Gdual$. Suppose $\mathcal{S}$ is $\chi$-unitary. Then, $\mathcal{S}^M$ is also $\chi$-unitary, i.e.
\begin{equation}\label{eq: level-M property V}
\sum_{\dlet\in\mathcal{D}} \chi(\spin(\mathcal{S}^{M}_{\vec{i}}(\dlet)))\,\overline{\chi(\spin(\mathcal{S}^{M}_{\vec{j}}(\dlet)))}=0,
\end{equation}
for all $\vec{i},\vec{j}\in \mathcal{D}^{(M)}$ whose corresponding $(M-1)$st digits in the $(Q,\mathcal{D})$-adic expansion are distinct.     
\end{prop}

\begin{proof}

We prove that Eq.~\eqref{eq: level-M property V} is satisfied if the level-$1$ version is. Let $(i^{(M-1)},\ldots,i^{(1)},i^{(0)})$ and $(j^{(M-1)},\ldots,j^{(1)},j^{(0)})$
be the $(Q,\digits)$-adic expansion of $\vec{i}$ and $\vec{j}$, respectively, where we assume that $i^{(M-1)}\neq j^{(M-1)}$. Inserting Eq.~\eqref{eq: spin level-M}, and using the multiplicativity of $\chi$, one gets
\begin{align*}
\alpha^{ }_{\vec{i}}\,\overline{\alpha^{ }_{\vec{j}}}\sum_{\dlet\in\mathcal{D}} \chi(W(\dlet,i^{(M-1)}))\,\overline{\chi(W(\dlet,j^{(M-1)}))}=0, 
\end{align*}
which follows from the level-$1$ version since $i^{(M-1)},j^{(M-1)}\in \mathcal{D}$ with $i^{(M-1)}\neq j^{(M-1)}$, thus proving the claim. Here, $\alpha^{ }_{\vec{i}}=\prod^{M-1}_{\ell=1}\chi(W(i^{(\ell)},i^{(\ell-1)}))$.
\end{proof}

\begin{proof}[Proof of Theorem~\textnormal{\ref{thm:main result}(b)}]
We first show that the associated spectral measure to each $f^{\chi}_{\dlet}$ is a constant multiple of Lebesgue measure. It is clear that $\widehat{f^{\chi}_{\dlet}}(0)=\frac{1}{|\digits|}$. We need to show that, for all $\vec{j}\neq 0$ and $\dlet=\dlet^{\prime}$,  \begin{equation}\label{eq: fourier cross correlation}
\big\langle
U_{\vec{j}}f^{\chi}_{\dlet} \mid f^{\chi}_{{\dlet}^{\prime}}
\big\rangle=\int_{\seqsp} f^{\chi}_{\dlet}(\mathcal{T}-\vec{j})\,\overline{f^{\chi}_{{\dlet}^{\prime}}(\mathcal{T})}\,\dd\mu(\mathcal{T})=0.
\end{equation} 
Note that it suffices to consider $\vec{j}\in Q\ZZ^{m}$, since, for any other $\vec{j}$, the functions $f^{\chi}_{\dlet}$ and $U_{\vec{j}}f^{\chi}_{\dlet}$ cannot simultaneously be non-zero due to property \RO.   
Let $\vec{j}\in\mathbb{Z}^m$ and ${\dlet}\in\mathcal{D}$ be fixed. We can restrict to the set of $\mathcal{T}\in \seqsp$ with $\mathcal{T}(0),\mathcal{T}(x)\in [\dlet]$. For each $M\in\mathbb{N}$, define 
$\seqsp^{(\dlet,{\dlet}^{\prime)}}(M)$ to be the set of all $\mathcal{T}\in \seqsp$ which satisfy $\mathcal{T}(0)\in [\dlet^{\prime}],\mathcal{T}(\vec{j})\in [\dlet]$, and for which $M$ is the smallest integer such that $\mathcal{T}(0)$ and $\mathcal{T}(\vec{j})$ are contained in the same level-$M$ supertile. Set $\seqsp(M):=\seqsp^{(\dlet,\dlet)}(M)$.
 
One then has
\[
\int_{\seqsp} \one^{ }_{[\dlet]}(\mathcal{T}-\vec{j})\one^{ }_{[\dlet]}(\mathcal{T})\dd\mu(\mathcal{T})=\sum_{M=1}^{\infty}\mu(\seqsp(M)).
\]
Now consider the following subset of $\mathcal{D}^{(M)}$:
\[
P(M)=\bigg\{\vec{p}\in \mathcal{D}^{(M)}\mid \vec{p}+\vec{j}\in \mathcal{D}^{(M)},\,  \vec{p}\text{ and }\vec{p}+\vec{j}\text{ are not in the same level-}(M-1)\text{ supertile} \bigg\}.
\]
Assuming that $\mathcal{S}^{M}_{\vec{p}}(\alet),\mathcal{S}^{M}_{\vec{p}+\vec{j}}(\alet)\in [\dlet]$,  with $\vec{p}\in P(M)$ and $\alet\in \mathcal{A}$, construct  the set \[\seqsp(\alet,M,\vec{p})=\big\{\mathcal{T}\in \seqsp(M)\mid \mathcal{T}(0)=\mathcal{S}^{M}_{\vec{p}}(\alet)\big\}.\]
This is the set of all elements of $\seqsp$ where the $M$-supertile at the origin is of type $\alet$ at position $\vec{p}$. 
This yields the decomposition
$\seqsp(M)=\bigcup_{\vec{p}\in P(M)}\bigcup_{\alet \in \mathcal{A}} \seqsp(\alet, M,\vec{p})$.

Going back to the Fourier coefficients, observe that
the functions $f^{\chi}_{\dlet}$ are constant over $\seqsp(\alet, M,\vec{p})$, where one has  $f^{\chi}_{\dlet}(\mathcal{T})=\chi\big(\spin(\mathcal{S}^{M}_{\vec{p}}(\alet))\big)$ for any $\mathcal{T}\in \seqsp(\alet, M,\vec{p})$. Using this decomposition and $f^{\chi}_{\dlet}(\mathcal{T}-\vec{j})=\chi\big(\spin((\mathcal{T}-\vec{j})(0))\big)\one^{ }_{[\dlet]}(\mathcal{T}-\vec{j})$, we find
\begin{align*}
\widehat{f^{\chi}_{\dlet}}(\vec{j})&=\sum_{M=0}^{\infty}\sum_{\vec{p}\in P(M)}\sum_{\alet\in\mathcal{A}}\int_{\seqsp(\alet,M,\vec{p})} \chi\big( \spin((\mathcal{T}-\vec{j})(0))\big)\, \overline{\chi\big(\spin(\mathcal{T}(0))\big)} \,\dd\mu(\mathcal{T})\\
&=\sum_{M=0}^{\infty}\sum_{\vec{p}\in P(M)}\sum_{\alet\in\mathcal{A}}
\chi\big(\spin(\mathcal{S}^{M}_{\vec{p}+\vec{j}}(\alet))\big)\,\overline{\chi\big(\spin(\mathcal{S}^{M}_{\vec{p}}(\alet))\big)} \,\mu(\seqsp(\alet,M,\vec{p}))\\
&=\sum_{M=0}^{\infty}\sum_{\vec{p}\in P(M)} C(M,\vec{p})\sum_{\dlet\in\mathcal{D}} \chi\big(
\spin(\mathcal{S}^{M}_{\vec{p}+\vec{j}}(\dlet))\big)\,\overline{\chi\big(\spin(\mathcal{S}^{M}_{\vec{p}}(\dlet))\big)}=0. 
\end{align*}
Note that the product $\chi\big(\spin(\mathcal{S}^{M}_{\vec{p}+\vec{j}}(\alet))\big)\,\overline{\chi\big(\spin(\mathcal{S}^{M}_{\vec{p}}(\alet))\big)}$ is constant for all $\alet \in [\dlet]$ by property \RT. Moreover, $\mu(\seqsp(\alet,M,\vec{p}))$ only depends on $M$ and $\vec{p}$. This means that there are $|G|$ copies of the same summand in the second equality, which implies the third one. There, one has $C(M,\vec{p})=|G|\mu(\seqsp(\dlet,M,\vec{p}))$, where $\mu(\seqsp(\dlet,M,\vec{p}))$ does not depend on $\dlet$.  The last equality then follows from Proposition~\ref{prop: level-M property V}.

We next show that different spin-free letters generate orthogonal cyclic subspaces, i.e., 
 $Z\big(f^{\chi}_{\dlet}\big)\perp Z\big(f^{\chi}_{\dlet^{\prime}}\big)$ for $\dlet\neq \dlet^{\prime}$ whenever $\sub$ is $\chi$-unitary. To show this, one has to prove that 
Eq.~\eqref{eq: fourier cross correlation} holds
for $\dlet\neq \dlet^{\prime}$.
One can then proceed as in the previous proof except that we  now consider $\seqsp^{(\dlet,\dlet^{\prime})}$ so that 
$\seqsp^{(\dlet,\dlet^{\prime})}=\bigcup_{\vec{p}\in P(M)}\bigcup_{\alet\in\mathcal{A}} \seqsp^{(\dlet,\dlet^{\prime})}(\alet,M,\vec{p})$.
Here, $\mathcal{S}^{M}_{\vec{p}}(\alet)\in  [\dlet^{\prime}]$ and $\mathcal{S}^{M}_{\vec{p}+\vec{j}}(\alet)\in [\dlet]$, and 
\[
\seqsp^{(\dlet,\dlet^{\prime})}(\alet,M,\vec{p})=\big\{\mathcal{T}\in \seqsp(M)\mid \mathcal{T}(0)=\mathcal{S}^{M}_{\vec{p}}(\alet)\text{ and } (\T-\vec{j})(0)=\mathcal{S}^{M}_{\vec{p}+\vec{j}}(\alet)\big\}.
\]
  One can check that $
\chi$-unitarity, together with this decomposition, lead to the inner product in Eq.~\eqref{eq: fourier cross correlation} being $\int_{\seqsp} f^{\chi}_\dlet(\mathcal{T}-\vec{j})\,\overline{f^{\chi}_{\dlet^{\prime}}(\mathcal{T})}\,\dd\mu(\mathcal{T})=0$, for all $\vec{j}\in\mathbb{Z}^m$, which proves the claim.
\end{proof}

\section{Diffraction and dynamical spectrum}\label{sec:diffraction}

\subsection{Spectral measures via diffraction measures}

Another spectral notion associated to tilings and point sets is the diffraction $\widehat{\gamma}$, which is a positive measure on $\mathbb{R}^m$. Given an element $\mathcal{T}$ of the subshift $\seqsp$, one can derive from it a colored point set $\varLambda_{\mathcal{T}}$ by choosing control points for each tile. We can pick the control point to be a vertex or the center of the cube such that the underlying point set becomes $\mathbb{Z}^m$ and $\varLambda_{\mathcal{T}}=\bigcup_{g\in G}\bigcup_{\dlet\in\digits} \varLambda_{g\dlet}$ . From $\varLambda_{\mathcal{T}}$, one can then build a weighted Dirac comb $\omega$ by considering a complex-valued weight function $w\colon \varLambda_{\mathcal{T}} \to \mathbb{C}$, which yields
$
\omega=\sum_{\vec{x}\in\mathbb{Z}^m} w(\vec{x})\delta_{\vec{x}}
$. 

In our situation, the autocorrelation of $\omega$ is known to be 
 \begin{equation}\label{eq: autocorrel coeff}
 \gamma=\sum_{\vec{j}\in\mathbb{Z}^m} \eta(\vec{j})\delta_{\vec{j}}, \text{ where }  \eta(\vec{j})=\lim_{N\rightarrow\infty} \frac{1}{(2N+1)^d}\sum_{|\vec{x}|\leqslant N} w(\vec{x})\overline{w(\vec{x}-\vec{j})}.
 \end{equation} 
 The existence of the  {\bf \em autocorrelation coefficients} $\eta(\vec{j})$ is guaranteed due to unique ergodicity.
The Fourier transform $\widehat{\gamma}$ is the {\bf \em diffraction measure}. Like spectral measures, the diffraction admits a decomposition into pure point, absolutely continuous, and singular continuous components. We refer the reader to \cite{BG} for a thorough introduction on diffraction theory.

It is well known that there is a connection betwen the spectral measure of functions in $L^{2}(\seqsp,\mu)$ and the admissible diffraction measures. In particular, for systems with pure point spectra, both notions are equivalent, with very little assumptions on ergodicity and local complexity \cite{BL}. For systems with mixed spectra, it is more complicated, and one usually needs a (possibly infinite) family of diffraction measures to characterize the entire dynamical spectrum \cite{BLvE}.

Recently, Lenz has shown in \cite{Lenz} that all spectral measures of a dynamical system with an action of a locally compact, $\sigma$-compact abelian group $G$ are recoverable as diffraction measures. 
Recall from above that the functions $f^{\chi}_{\dlet}$ generate $L^{2}(\seqsp,\mu)$. We show below that there is an explicit weighted Dirac comb whose diffraction is exactly the spectral measure $\sigma_{f^{\chi}_{\dlet}}$. To this end, we need the following notions; see \cite{Lenz}. 

Let $(\seqsp,\ZZ^m,\mu)$ be an ergodic dynamical system and $U$ be the corresponding Koopman operator for the $\ZZ^m$-action.
Let $f\in L^{2}(\seqsp,\mu)$. Consider the map $\mathcal{N}\colon C_{\text{c}}(\ZZ^m)\to L^{2}(X,\mu)$ given by 
$\mathcal{N}^{f}(\phi):=\sum_{\vec{x}\in\ZZ^m} \phi(\vec{x})U_{\vec{x}}f$,
for $\phi\in C_{\text{c}}(\ZZ^m)$, which is linear and $U$-equivariant.
To $\mathcal{N}^{f}$, one can canonically associate a diffraction measure $\widehat{\gamma^{ }_{f}}$, where the autocorrelation $\gamma^{ }_{f}$ is of the form $\gamma^{ }_{f}=\sum_{\vec{j}\in\ZZ^m} \eta(\vec{j})\delta_{\vec{j}}$, with 
\[
\eta(\vec{j})=\eta_f(\vec{j}):= \left\langle U_{\vec{j}}\mathcal{N}^{f}(\one_{0})\mid\mathcal{N}(\one_0) \right\rangle.
\]
Here, $\left\langle \cdot\mid\cdot\right\rangle$ is the inner product in $L^2(\seqsp,\mu)$ and $\one_0$ is the lookup function at the origin. One then has the following result.

\begin{theorem}[{ \cite[Thm.~4]{BLvE},\cite[Thm.~2.1]{Lenz}}]\label{thm:diffraction dynamical}
Let $f\in L^{2}(\seqsp,\mu)$. Then, the spectral measure $\sigma_f$ of $f$ is the diffraction measure $\widehat{\gamma^{ }_f}$. \qed
\end{theorem}

In particular, when $f\in L^{2}(\seqsp,\mu)$ is a hyperlocal function, i.e., $f(\tiling)$ only depends on $\tiling(0)$, the spectral measure $\sigma_f$ is the diffraction 
of the weighted comb $\omega=\sum_{\vec{x}\in \Z^m} f(\tiling(\vec{x}))\delta_{x}$, for \emph{any} $\tiling\in\seqsp$ whenever $(\seqsp,\Z^m)$ is strictly ergodic, which always holds for primitive qubit substitutions; see Proposition~\ref{prop: strict ergod}.

\begin{remark}
Strictly speaking, the diffraction of a weighted Dirac comb supported on $\ZZ^m$ will be a measure on $\mathbb{R}^m$ which is $\ZZ^m$-periodic, i.e., $\widehat{\gamma}=\delta_{\ZZ^m}\ast \widehat{\gamma}^{ }_{\textnormal{FD}}$, where $\widehat{\gamma}^{ }_{\textnormal{FD}}$ is a measure on $\mathbb{T}^m$ and is called the \emph{fundamental diffraction} in \cite{BLvE}. By abuse of notation, and to be consistent with Theorem~\ref{thm:diffraction dynamical}, we will simply refer to $\widehat{\gamma}^{ }_{\textnormal{FD}}$ as the diffraction. \exend
\end{remark}

\begin{example}
Consider the Thue--Morse substitution $\varrho\colon \alet\mapsto \alet\blet, \blet\mapsto \blet\alet$ and choose $f$ to be $f=\one_{\alet}(\mathcal{T}(0))\in L^{2}(\seqsp,\mu)$. It follows from Theorem~\ref{thm:diffraction dynamical} that the spectral measure $\sigma_f$ is the diffraction of the weighted Dirac comb on any element on the Thue--Morse hull with weights $w_{\alet}=1$ and $w_{\blet}=0$. There is a closed form for the diffraction of such a Dirac comb for Thue--Morse, which is given by 
\[
\widehat{\gamma^{ }_\omega}=\left|\frac{w_{\alet}+w_{\blet}}{2}\right|^2\delta_0+\left|\frac{w_{\alet}-w_{\blet}}{2}\right|^2\widehat{\gamma^{ }_{\text{TM}}},
\]
where $\widehat{\gamma^{ }_{\text{TM}}}$ is the purely singular continuous Thue--Morse measure; see \cite{BG}. 
Substituting $w_{\alet}=1$ and $w_{\blet}=0$ yields $\sigma_f=\frac{1}{4}\delta_0+\frac{1}{4}\widehat{\gamma^{ }_{\text{TM}}}$; compare \cite{Queffelec2}. \exend
\end{example}

In the setting of qubit spin substitutions, we have the following correspondence. 

\begin{example}\label{ex:qubit diffraction dynamical}

Let $\sub$ be a spin substitution with spin group $G$ and digit set $\digits$. 
Consider the function $f=f^{\chi}_{\dlet}:=\chi(\spin(\mathcal{T}(0)))\one_{[\dlet]}$. Here one has $U_{\vec{x}}\mathcal{N}^{f}(\one_0)=
\chi(\spin(\mathcal{T}-\vec{x}))U_{\vec{x}}\one_{[\dlet]}$. Moreover, the diffraction $\widehat{\gamma^{ }_{f}}$ is the diffraction of the weighted Dirac comb $\omega_f=\sum_{\vec{x}\in\ZZ^m}\omega(\vec{x})\delta_{\vec{x}}$, where the weights are given by 
\[
\omega(\vec{x}):=\begin{cases}
\chi(\spin(\mathcal{T}-\vec{x})),& \text{if } \dig(\mathcal{T}-\vec{x})\in[\dlet],\\
0, & \text{otherwise}.  
\end{cases}
\]
From Theorem~\ref{thm:diffraction dynamical} and the discussion before it, one has $\sigma_{f^{\chi}_\dlet}=\widehat{\gamma^{ }_f}$.
\exend
\end{example}

When $\sub$ is $\chi$-unitary,  we recover the following implication for the diffraction measure $\widehat{\gamma}$ for a specific choice of weight functions. 

\begin{coro}\label{coro:diffraction ac}
Let $\varLambda$ be a colored point set arising from a qubit spin substitution.
Consider the weight functions $w^{\chi}:\varLambda\to \mathbb{C}$, which are defined via $w^{\chi}(\vec{x})=\chi(\spin(\mathcal{T}(\vec{x})))$, and the corresponding weighted Dirac comb $\omega^{\chi}=\sum_{\vec{x}\in\varLambda} w^{\chi}(\vec{x})\delta_{\vec{x}}$. Then, if $\mathcal{S}$ is $\chi$-unitary, the associated diffraction $\widehat{\gamma}$ is  Lebesgue measure $\mu^{ }_{\textnormal{L}}$ in $\mathbb{R}^m$. 
\end{coro}

\begin{proof}
Note that, from Theorem~\ref{thm:diffraction dynamical}, the diffraction $\widehat{\gamma}$ of $\omega^{\chi}$ is the spectral measure 
of the function $F^{\chi}=\sum_{\dlet\in \digits} f^{\chi}_{\dlet}$. From Theorem~\ref{thm:main result}(b),
the functions $f^{\chi}_{\dlet}$ define orthogonal cyclic subspaces, and hence one has $\sigma^{ }_{F^{\chi}}=\sum_{\dlet\in\digits} \sigma_{f^{\chi}_{\dlet}}$.
The claim follows since each of the summands is $\frac{1}{|\digits|}\mu^{ }_{\text{L}}$.
\end{proof}

Theorem~\ref{thm:diffraction dynamical} allows one to derive restrictions on the maximal spectral type of $f^{\chi}_{\dlet}$ by looking at the possible spectral types of $\widehat{\gamma^{ }_f}$, which we discuss in more detail in the next sections. 

\subsection{Substitutive factors and singular continuous spectrum}\label{sec:factors-sc}
Here, we describe how non-trivial characters $\chi\in \Gdual$ reveal factors of $(\seqsp,\Z^m)$ which are relevant in determining the spectral type of $H^{\chi}$
when $\sub$ is not $\chi$-unitary.

\begin{prop}\label{prop: factors}
Let $\sub$ be an aperiodic primitive spin qubit substitution in $\Z^m$, with spin matrix $W$ and spin group $G$. For each $\chi\in \Gdual$, there exists a substitutive factor $\sub^{(\chi)}$ over the alphabet ${\mA}^{(\chi)}=G/\textnormal{ker}(\chi)\times \digits$. Moreover, the spectral type of $H^{\chi}$ is absolutely continuous with respect to the maximal spectral type of $\sub^{(\chi)}$. 
\end{prop}

\begin{proof}
Let $g\dlet, g^{\prime}\dlet\in \mA$. If $g$ and $g^{\prime}$ belong to the same coset of $\text{ker}{(\chi)}$, the corresponding supertiles satisfy 
\[
\chi(\sub(g\dlet))=\chi(\sub(g^{\prime}\dlet))
\]
by Property \RT. This gives rise to an equivalence relation in $\mA$ 
given by $g\dlet\sim g^{\prime}\dlet \iff g\text{ker}(\chi)=g^{\prime}\text{ker}(\chi)$. 
 Using this relation, one can now define the factor qubit substitution $\sub^{(\chi)}$ via the matrix $W^{\prime}=\chi(W)$. Since $G/\textnormal{ker}(\chi)$ is cyclic \cite{Meshulam}, the entries of $W^{\prime}$
are now roots of unity which generate the cyclic group $G/\text{ker}(\chi)$.
The factor map is given by $\pi\colon \seqsp\to \seqsp^{(\chi)}$ with $\pi(\tiling(\vecj))=\chi(\spin(\tiling(\vecj))\dig(\tiling(\vecj))$
for $\vecj\in\Z^m$.  
The last claim follows since we have the equivalence 
$\chi(\spin(\tiling(0)))=\spinfac(\tiling^{\prime}(0))$, for some $\tiling^{\prime}$ in the factor subshift generated by $\sub^{(\chi)}$, and hence the spectral measure of $\sum_{\dlet \in\digits}f^{\chi}_{\dlet}\in H^{\chi}$ is the spectral measure of a function defined over the factor subshift $\seqsp^{(\chi)}$ which forgets the digits and retains the spin component in $G/\text{ker}(\chi)$. This specific spectral measure, by the discussion in Example~\ref{ex:qubit diffraction dynamical}, is recoverable from a diffraction measure of the factor substitution $\sub^{(\chi)}$. 
\end{proof}

The following is a consequence of Proposition~\ref{prop: factors} and Proposition~\ref{prop:purity-subshift}. 

\begin{coro}\label{coro:pp-sc}
If the factor substitution $\sub^{(\chi)}$ has singular spectrum, then either $H^{\chi}$ is pure point or purely singular continuous. \qed
\end{coro}

We note that these factor substitutions exist for all qubit spin substitutions. When the spin matrix $W$ has more structure, it is possible
that letters in $\mA$ with \emph{different} digits get identified as well and hence $\sub^{(\chi)}$ becomes a substitution on an alphabet smaller than $G/\text{ker}(\chi)\times\digits$. In these cases, they may no longer be spin substitutions but will remain qubit substitutions supported on the same digit set.

We now show that, when $\chi(W)$ is a rank-1, one obtains a much simpler factor substitution. We recall that a {\bf\emph{bijective qubit substitution}} is one where the map $\sub_{\vec{d}}\colon \mA\to \mA$
is a bijection for all $\vec{d}\in\digits$.

\begin{prop}\label{prop:rank1}
Let $\sub$ be as above. If $\chi(W)$ is a rank-1 matrix, the associated factor substitution $\sub^{(\chi)}$ is a bijective qubit substitution
over the alphabet $G/\textnormal{ker}(\chi)$. 
\end{prop}

\begin{remark}
The reader may wish to gain intuition by considering, for instance, the Vierdrachen example in Example~\ref{ex:vierdrachen} with character $\chi_2\in \widehat{C_2\times C_2}$. The proof below makes precise the identifications that one sees in such examples. \exend
\end{remark}

\begin{proof}
Since $\chi(W)$ is rank-$1$, there exist vectors $v,w\in (S^{1})^{|\digits|}$ which satisfy $\chi(W)=v\otimes w$. This induces another equivalence relation on the alphabet as follows. Let $\dlet_i,\dlet_j$ be spin-free letters in $\mA$ which correspond to the rows $W_i$ and $W_j$ of the spin matrix. Then, one has $\dlet_i\sim g\dlet_j \iff \chi(W_i)=\chi(g)\chi(W_j)$. This relation induces a partition of the alphabet satisfying the following properties.
\begin{enumerate}
\item The set of equivalence classes is indexed by $G/\text{ker}(\chi)$. More explicitly, we have \[
\mathcal{A}=\bigcup_{g\in G/\text{ker}(\chi)}\,[g\dlet_0],\]
where each equivalence class corresponds to a coset of $\text{ker}(\chi)$. 
\item Because $\chi(W)$ is rank-1, each equivalence class contains at least one element of $[\dlet]$ for each $\dlet\in\digits$. 
\item It is possible for two letters with the same spin to be in the same equivalence class (this happens when $\chi(W_i)=\chi(W_j)$).   
\end{enumerate}
The substitution $\sub^{(\chi)}$ induced by $\chi$ on the equivalence classes is given by 
\begin{equation}\label{eq:subs induced}
\sub^{(\chi)}\colon [\dlet_0]\mapsto [g_0\dlet_0][g_1\dlet_1]\cdots[g_{\ell-1}\dlet_{\ell-1}]
\end{equation}
where $\ell=|\digits|$. From the second statement, we know that each equivalence class $[g_{j}\dlet_{j}]$ contains at least one element whose spin-free projection is $\dlet_0$, i.e., $[g_{j}\dlet_{j}]=[g^{\prime}_{j}\dlet_{0}]$ for some $g^{\prime}_{j}$. Using $[g\dlet_j]=g[\dlet_j]$, Eq.~\eqref{eq:subs induced} then simplifies to
\[\sub^{(\chi)}\colon [\dlet_0]\mapsto g^{ }_0[\dlet_0]\,g^{\prime}_1[\dlet_0]\cdots g^{\prime}_{\ell-1}[\dlet_{0}],\] which one can then write as a substitution on the alphabet $G/\text{ker}(\chi)$ as
\[
\sub^{(\chi)}\colon g\mapsto L_{g_0}(g) L_{g^{\prime}_1}(g)\cdots  L_{g^{\prime}_{\ell-1}}(g)
\]
for $g\in G/\text{ker}(\chi)$, where $L_g$ is left multiplication by $g$.
Our alphabet now is $G/\text{ker}(\chi)$ and
one can easily check that $\sub^{(\chi)}$ is a bijective substitution whose $1$-tiles are still supported on $\digits$. 
\end{proof}

\begin{proof}[Proof of Theorem~\textnormal{\ref{thm:main result}(c)}]
Since $G/\text{ker}(\chi)$ is cyclic \cite{Meshulam}, the factor substitution $\sub^{(\chi)}$ is abelian. Moreover, $\sub^{(\chi)}$ is primitive, because $\sub$ is primitive, and hence all equivalence classes $[g\dlet_0]$ appear in $\big(\sub^{(\chi)}\big)^n$ for $n $ large enough. 
It then follows from several singularity results for bijective abelian substitutions \cite{Bartlett, BGM} that $\sub^{(\chi)}$ can only give rise to singular diffraction for arbitrary weights on the alphabet, and hence singular spectral measures for functions in $H^{\chi}$. It is also well known that the measures which generate the maximal spectral type of $\sub^{(\chi)}$ are Riesz products, from which the second claim follows. The last claim follows from a higher-dimensional generalization of Dekking's result, i.e., an aperiodic bijective substitution on $\ZZ^{m}$ must have a non-trivial continuous component (which in this case is singular continuous); compare \cite{Frank3}. 
\end{proof}

\begin{remark}
In the setting of Corollary~\ref{coro:pp-sc}, either $H^{\chi}$ is singular continuous or pure point. Both cases are possible; see Examples~\ref{ex:vierdrachen} and \ref{ex:pp-factor}. \exend
\end{remark}

If $\chi(W)$ is not of rank 1, one can still analyze the spectral type of $\sub^{(\chi)}$ via several singularity criteria. 
For certain subclasses, one can apply the following sufficient criterion for singularity due to Berlinkov and Solomyak.

\begin{theorem}[{\cite[Thm.~1.1]{BS}}]\label{thm: Berlinkov-Solomyak}
Let $\varrho$ be a one-dimensional substitution of constant-length $L$. Then, if  the substitution matrix $M$ of $\varrho$ does not have an eigenvalue of modulus $\sqrt{L}$, the dynamical spectrum of $\varrho$ is purely singular.  \qed
\end{theorem}

In other cases one can prove singularity results for the factor substitution $\sub^{(\chi)}$ using Lyapunov exponents, which we discuss next.

\subsection{Analysis via Lyapunov exponents}
\label{sec:lyapunov}

An important quantity that reveals the scaling behavior and spectral type of diffraction measures of substitution dynamical systems is the Lyapunov exponent of the associated matrix cocycle; see \cite{BGM,BuS}. Below, we present the properties of the matrix cocycle associated with qubit spin substitutions.

Let $\mathcal{S}$ be a qubit substitution with digit set $\mathcal{D}$, spin group $G$ and map $W$. 
Let $\alet_i=g\dlet$ and  $\alet_j=g^{\prime}\dlet^{\prime}$, with $g,g^{\prime}\in G$ and $\dlet,\dlet^{\prime}\in \mathcal{D}$.
One has the decomposition $\mathcal{D}=\bigcup T_{ij}$, where the sets $T_{ij}$ consists of the positions of tiles of type $\alet^{ }_i$ in $\mathcal{S}(\alet_{j})$. Note that every level-1 supertile contains at most one of any of the variations of each basic tile type, which means $\text{card}(T_{ij})\in\big\{0,1\big\}$, for all $1\leqslant i,j\leqslant |\digits||G|$. One can build the matrix $T=(T_{ij})$, which we call the  \emph{displacement matrix}.  
The displacement matrix $T$ and the expansive map $Q$ completely determine the hierarchical structure of $\mathcal{S}$, i.e., the structure of the level-$n$ supertiles. 
In general, one has 
\[
\mathcal{S}^n(\alet_{j})=\bigcup_{m=1}^{|\digits||G|} \left(QT^{(n-1)}_{mj}+\mathcal{S}(\alet_m)\right),
\]
where $T^{(1)}=T$ and $T^{(n)}$ is the displacement matrix for the power $\mathcal{S}^{n}$. Note that the matrix of sets $T^{(n)}$ is related to $\mathcal{D}^{(n)}$ via $\mathcal{D}^{(n)}=\bigcup_{ij} T^{(n)}_{ij}$.

 From this combinatorial information, one can build another matrix $B(\vec{k})$, which is called the Fourier matrix, via $B(\vec{k})_{ij}:=\sum_{\vec{x}\in T_{ij}} \ee^{2\pi\ii \left\langle \vec{k}\mid\vec{x}\right\rangle}$. Equipped with the base dynamics $ \vec{k}\mapsto Q^{\intercal}\vec{k}$, this matrix-valued function becomes a cocycle $
B^{(n)}(\vec{k})=B(\vec{k})B(Q^{\intercal} \vec{k})\cdots B((Q^{\intercal})^{n-1}\vec{k})
$. 
The {\bf \em (maximal) Lyapunov exponent} of this cocycle is given by 
\[
\lambda^{B}(\vec{k}):=\limsup_{n\rightarrow\infty}\log\|B^{(n)}(\vec{k})\|
\] 
for $\vec{k}\in \mathbb{R}^m.$

\begin{example}[Triomino substitution]
Recall the Triomino substitution from Example~\ref{ex:triomino}. 
Fix an ordering of $C_3$ as $\left\{1,\omega,\omega^2\right\}$.
From the $C_3$-invariance of $\mathcal{S}$, one can directly verify that the displacement matrix is given by
$T=\phi(1)\otimes T_1 + \phi(\omega)\otimes T_{\omega} +\phi(\omega^2)\otimes T_{\omega^2}$ with 
\[
T_1=\begin{pmatrix}
0 & 0 & 0 \\
e^{ }_1 & \varnothing & \varnothing\\
e^{ }_2 & \varnothing & \varnothing
\end{pmatrix}\quad \quad T_{\omega}=\begin{pmatrix}
\varnothing & \varnothing & \varnothing \\
\varnothing & \varnothing & e^{ }_1\\
\varnothing & e^{ }_2 & \varnothing
\end{pmatrix} \quad \quad
T_{\omega^2}=\begin{pmatrix}
\varnothing & \varnothing & \varnothing \\
\varnothing & e^{ }_1 & \varnothing \\
\varnothing & \varnothing  & e^{ }_2
\end{pmatrix}
\]
and $\phi\colon C_3\to \text{GL}(3,\ZZ)$ is the representation via permutation matrices.
It then follows that, for this substitution, the Fourier matrix is given by
\begin{equation}\label{eq: Fourier matrix - TripleRS}
B(\vec{k})=\phi(e)\otimes Z_e + \phi(\sigma)\otimes Z_g +\phi(\sigma^2)\otimes Z_{g^2} \quad\quad \text{with}
\end{equation}
\[
Z_e=\begin{pmatrix}
1 & 1 & 1 \\
x & 0 & 0\\
y & 0 & 0
\end{pmatrix}\quad \quad Z_g=\begin{pmatrix}
0 & 0 & 0 \\
0 & 0 & x\\
0 & y & 0
\end{pmatrix} \quad \quad
Z_{g^2}=\begin{pmatrix}
0 & 0 & 0 \\
0 & x & 0 \\
0 & 0  & y
\end{pmatrix},
\]
where $x=\ee^{2\pi\ii \left\langle \vec{k}\mid e_1\right\rangle}$ and $y=\ee^{2\pi\ii \left\langle \vec{k}\mid e_2\right\rangle}$ with $\vec{k}\in\mathbb{R}^2$.   \exend
\end{example}

The following result relates $\lambda^{B}(\vec{k})$ to the presence of absolutely continuous diffraction.

\begin{theorem}[\cite{BGM}]\label{thm: presence of ac}
Let $\mathcal{S}$ be a primitive stone inflation in $\mathbb{R}^m$, with expansive map $Q$. Suppose $\det(B(\vec{k}))\not\equiv 0$. 
If the diffraction $\widehat{\gamma}$ has a non-trivial absolutely continuous component, the Lyapunov exponent $\lambda^{B}(\vec{k})$ associated to the Fourier cocycle generated by $B(\vec{k})$ with base dynamics $Q^{\intercal}$ must satisfy 
\[
\lambda^{B}(\vec{k})=\log\sqrt{|\det Q|},
\]
for a set of $\vec{k}\in\mathbb{R}^m$ of positive Lebesgue measure. \qed 
\end{theorem}

We have the following complementary condition which proves the singularity of the diffraction; see \cite{BuS} for the analogue for spectral measures. 

\begin{theorem}[\cite{BGM}]\label{thm: absence of ac}
Let $\mathcal{S}$ be a primitive stone inflation in $\mathbb{R}^m$, with expansive map $Q$. Suppose $\det(B(\vec{k}))\not\equiv 0$. 
If there exists an $\varepsilon >0$ such that the Lyapunov exponent $\lambda^{B}(\vec{k})$ associated to the Fourier cocycle generated by $B(\vec{k})$ with base dynamics $Q^{\intercal}$ satisfies
\[
\lambda^{B}(\vec{k})<\log\sqrt{|\det Q|}-\varepsilon,
\]
for a.e. $\vec{k}\in\mathbb{R}^m$, then, for any choice of complex weights, the diffraction $\widehat{\gamma}$ is purely singular with respect to Lebesgue measure. \qed 
\end{theorem}

\begin{remark}
Note that qubit substitutions are generally not stone inflations because the level-1 supertiles are not merely expanded squares. Nevertheless, since it arises from a primitive inflate-and-subdivide rule, the pair correlations are well defined as limits of Birkhoff averages and are constant in $\seqsp$, and hence the arguments involving Lyapunov exponents in Theorems~\ref{thm: presence of ac} and \ref{thm: absence of ac} follow through for these cases; compare \cite[Rem.~5.6]{BGM}. \exend
\end{remark}

We now present the general structure of Fourier cocycles associated to qubit substitutions, and prove that they satisfy a modified version of Theorem~\ref{thm: presence of ac} whenever $\sub$ is $\chi$-unitary. This is consistent with the presence of absolutely continuous components  shown in Theorem~\ref{thm:main result}(b). 

\begin{prop}\label{prop: Lyapunov exponents Vandermonde}
Let $\mathcal{S}$ be an aperiodic primitive qubit spin substitution in $\ZZ^m$ given by $(Q,\digits,G,W)$.
The Fourier matrix $B(\vec{k})$ is unitarily block diagonalizable into 
\[
B(\vec{k})\cong \bigoplus_{\chi\in \Gdual} B^{ }_{\chi}(\vec{k}),
\]  
for all $\vec{k}\in\mathbb{R}^m$. 
Moreover, if $\sub$ is $\chi$-unitary, the subblock $B^{ }_{\chi}$ is of the form $\sqrt{|\det(Q)|}\,U$, where $U$ is a unitary matrix, which implies $
\lambda^{B^{ }_{\chi}}=\log\sqrt{|\det(Q)|}$
for all $\vec{k} \in\mathbb{R}^m$.  
\end{prop}

\begin{proof} 
It follows from the construction and from Property~\textbf{(R2)} of the substitution $\mathcal{S}$ that the Fourier matrix $B(\vec{k})$ is of the form 
\begin{equation}\label{eq: general form B}
B(\vec{k})=\sum_{g\in G}\phi(L_g)\otimes Z_g
\end{equation}
where $\phi:S_{|G|}\to\text{GL}(d,\mathbb{Z})$ is the permutation representation, and $\sum Z_g$ is a rank-1 matrix whose rows are of the form 
$r^{ }_\ell=x^{ }_\ell(1,\ldots,1)$, with $x_\ell=\ee^{2\pi\ii \left\langle \vec{k}\mid \vec{d}\right\rangle}
$ for some $\vec{d}\in \mathcal{D}$.

For the proof of the block diagonal structure, one invokes the property of the permutation matrices $\phi(L_g)$ being simultaneously diagonalisable. The subspaces $V_j=v_j\otimes \mathbb{C}^{|\digits|}$ are invariant with respect to $B(\vec{k})$, where $\left\{v_j\right\}$ are the (normalised) shared eigenvectors of the matrices $\left\{\phi(L_g)\right\}$. One can then choose the columns of the similarity transformation $S$ to be $v_j\otimes e^{ }_{\ell}$, where $e^{ }_\ell$ are the canonical basis vectors of $\mathbb{C}^{|\digits|}$. The corresponding blocks will then be of the form 
$B_{\chi}(\vec{k})=\sum_{g\in G} \chi(g) Z_g(\vec{k})$, for $\chi\in \Gdual$.

To prove the last claim, note that $B_{\chi}(\vec{k})$ also admits the representation
\[
B_{\chi}(\vec{k})=\begin{pmatrix}
p^{ }_1(\vec{k})\chi(W)^{ }_1\\
p^{ }_2(\vec{k})\chi(W)^{ }_2\\
\vdots \\
p^{ }_K(\vec{k})\chi(W)^{ }_K,
\end{pmatrix}
\]
where $p^{ }_\ell(\vec{k})=\ee^{2\pi\ii \left\langle\vec{k}\mid \vec{d}_\ell\right\rangle}$ with $\vec{d}_\ell\in \mathcal{D}$ and $\chi(W)^{ }_\ell$ is the $\ell$th row of the matrix $\chi(W)$. 
From the unitarity of $\frac{1}{\sqrt{|\digits|}}\chi(W)$ we get
\[
\big(B^{ }_{\chi}(\vec{k})B^{\dag}_{\chi}(\vec{k})\big)_{ij}=p^{ }_{i}(\vec{k})\,\overline{p}^{ }_{j}(\vec{k})|\digits|\delta_{ij}
\]
which yields $B^{ }_{\chi}(\vec{k})B^{\dag}_{\chi}(\vec{k})=|\digits|\mathbb{I}_{|\digits|}$ for all $\vec{k}\in\mathbb{R}^m$. As a direct consequence, we get that the cocycle  induced by this block satisfies $B^{(n)}_\chi(\vec{k})=|\digits|^{\frac{n}{2}}U^{(n)}(\vec{k})$, where $U^{(n)}(\vec{k})$ is a unitary matrix for all $\vec{k}\in\mathbb{R}^m$, which further implies that $B_\chi(\vec{k})$ has trivial Lyapunov spectrum, i.e., all exponents exist and are equal to $\lambda=\log\sqrt{|\digits|}$.
\end{proof}

\begin{example}\label{ex: Triomino Lyapunov exponents}
For the Triomino substitution in Example~\ref{ex:triomino}, we get that the Fourier matrix $B(\vec{k})$ can be block-diagonalised into the blocks
\[
B_{\chi^{ }_0}(\vec{k})=\begin{pmatrix}
1 & 1 & 1\\
x & x & x\\
y & y & y 
\end{pmatrix} \quad B_{\chi^{ }_1}(\vec{k})=\begin{pmatrix}
1 & 1 & 1\\
x & \omega x & \omega^2 x\\
y & \omega^2 y & \omega y 
\end{pmatrix} \quad \quad B_{\chi^{ }_2}(\vec{k})=\begin{pmatrix}
1 & 1 & 1\\
x & \omega^2 x & \omega x\\
y & \omega y & \omega^2 y 
\end{pmatrix},
\]
It can easily be checked that $\sub$ is  $\chi_1$- and $\chi_2$-unitary, which by Proposition~\ref{prop: Lyapunov exponents Vandermonde} implies that $B(\vec{k})$ has Lyapunov exponent $\lambda=\log\sqrt{3}$ of multiplicity 6.
In this case, we have exactly $(|\digits|-1)$ blocks with this exponent. \exend
\end{example}

\begin{remark}[Non-invertibility of $B(\vec{k})$]
Note that, in Example~\ref{ex: Triomino Lyapunov exponents}, $B_{\chi^{ }_0}(\vec{k})$ is a rank-1 matrix, which implies that, for all $\vec{k}$, $B_{\chi^{ }_0}(\vec{k})$ (and consequently $B(\vec{k})$) has zero as an eigenvalue of multiplicity at least $|\digits|-1$. This \textit{a posteriori} proves that $\det(B(\vec{k}))=0$ for all $\vec{k}\in\mathbb{R}^m$ for any qubit spin substitution $\mathcal{S}$, so Theorem~\ref{thm: presence of ac} cannot directly be invoked. However, when one restricts to the invariant subspace acted upon by $B_\chi(\vec{k})$, one recovers an analogous criterion which is satisfied by all of these examples. This is exactly the same mechanism for Rudin--Shapiro, which belongs to this family of substitution rules; see \cite{Manibo}. \exend
\end{remark}

As mentioned in the previous remark, Proposition~\ref{prop: Lyapunov exponents Vandermonde} suggests that our tilings could have absolutely continuous diffraction spectrum, which is something we directly confirm as a corollary of our main result regarding the spectral measures in the next section. 
Note however that Proposition~\ref{prop: Lyapunov exponents Vandermonde} is interesting in its own right because so far, we are not aware of an example of a substitution tiling with expansive map $Q$ with Lyapunov exponent equal to $\lambda=\log\sqrt{|\det(Q)|}$ that \emph{does not} contain an absolutely continuous spectral component, which suggests that a variant of the criterion in Theorem~\ref{thm: presence of ac} might be sufficient as well. 

The growth rate of the norm of the blocks $B_{\chi}(\vec{k})$ prescribes the scaling behavior of the spectral measures in $H^{\chi}$. It also determines the properties of diffraction measures where the weight functions are given by the characters in $\Gdual$.  Together with Theorem~\ref{thm:diffraction dynamical}, one can look at the Lyapunov exponents of $B_{\chi}(\vec{k})$ to rule out the presence of absolutely continuous components 
in $H^{\chi}$ when 
$\frac{1}{|\digits|}\chi(W)$ is not unitary but is still of full rank. 

\begin{prop}\label{prop: LE full rank}
Let $\sub$ be an aperiodic primitive spin substitution in $\ZZ^{m}$ with spin matrix $W$. If $\chi(W)$ is of full rank, $B_{\chi}(\vec{k})$ is invertible for a.e. $\vec{k}\in\mathbb{R}^m$. 
Moreover, the Lyapunov exponent $\lambda^{B_{\chi}}(\vec{k})$ exists as a limit and is constant for a.e. $\vec{k}\in\mathbb{R}^{m}$. 
\end{prop}

\begin{proof}
Almost everywhere invertibility follows from $\chi(W)=B_{\chi}(0)$ being of full rank, and from the analyticity of $B_{\chi}(\vec{k})$.
Moreover, $B(\vec{k})$ is $\mathbb{Z}^m$-periodic, i.e., $B(\vec{k})=B(\vec{j}+\vec{k})$ for $\vec{j}\in\Z^m$ and $\vec{k}\in\mathbb{R}^m$. 
Since the map $\vec{k}\mapsto Q^{\intercal}\vec{k}$ is ergodic with respect to the Haar measure of $\mathbb{T}^m$, the almost sure existence of the Lyapunov exponent follows from Kingman's subadditive ergodic theorem; see \cite{BP,Viana}. 
\end{proof}

\begin{prop}\label{prop:Lyapunov-spin-singular}
Let $\sub$ be as in Proposition~\textnormal{\ref{prop: LE full rank}} and let $\chi\in\Gdual$.
If the almost sure value of the Lyapunov exponent $\lambda^{B_{\chi}}$ is strictly less than $\log\sqrt{|\digits|}$, the spectral measure $\sigma_f$ is singular for all $f\in H^{\chi}$ and is either pure point or purely singular continuous.
\end{prop}

\begin{proof}
Let $\T\in \seqsp$ and let $\omega_{\chi}=\sum_{\vec{j}\in\Z^m} \chi(\spin(\T-\vec{j}))\delta_{\vec{j}}$ and let $\widehat{\gamma}:=\widehat{\gamma}_{\chi}$ be the corresponding (fundamental) diffraction measure. From Example~\ref{ex:qubit diffraction dynamical}, we know that $\widehat{\gamma}$ is the spectral measure $\sigma_f$ of the function $f=\sum_{\dlet\in\digits}f^{\chi}_{\dlet}$. From the spectral purity result in Proposition~\ref{prop:purity-subshift}, it suffices to determine the spectral type of $\sigma_f$ to know the spectral type of $H^{\chi}$. The idea is then to use the block diagonal structure of the Fourier matrix in Proposition~\ref{prop: Lyapunov exponents Vandermonde} and a slight modification of the proof of Theorem~\ref{thm: absence of ac}. 

Let $\widehat{\gamma}^{ }_{\textsf{ac}}=h(\vec{k})\mu_{\text{L}}$ be the absolutely continuous component of the diffraction (now viewed on $\mathbb{R}^m$), where $\mu_{\text{L}}$ is Lebesgue measure in $\R^m$ and $h(\vec{k})$ is the corresponding Radon--Nikodym density. One can decompose $h(\vec{k})$ into 
\[
h(\vec{k})=\sum_{\dlet,\dlet^{\prime}\in\mathcal{D}}\sum_{g,g^{\prime}\in G}\chi(g)\overline{\chi(g^{\prime})}h_{\alet,\alet^{\prime}}(\vec{k})
\]
where $\alet=g\dlet$ and $\alet^{\prime}=g\dlet^{\prime}$. The functions $h_{\alet,\alet^{\prime}}(\vec{k})$ satisfy $h_{\alet,\alet^{\prime}}(-\vec{k})=h_{\alet^{\prime},\alet}(\vec{k})=\overline{h_{\alet,\alet^{\prime}}(\vec{k})}$. This allows one to rewrite $h_{\alet,\alet^{\prime}}(\vec{k})$ as 
\[
h_{\alet,\alet^{\prime}}(\vec{k})=\sum_{\ell=1}^{s}\chi(g)(v_{\ell}(\vec{k}))_{\alet}(v_{\ell}(\vec{k}))^{T}_{\alet^{\prime}}\overline{\chi(g^{\prime})}
\] 
It then suffices to look at the growth rate of the vectors in the subspace \[V_{\chi}=(\chi(g^{ }_1),\ldots,\chi(g^{ }_{|G|}))\otimes \mathbb{C}^{|\digits|}\] under the matrix cocycle $B^{(n)}(\vec{k})$. From Proposition~\ref{prop: Lyapunov exponents Vandermonde}, we know that the restriction of $B(\vec{k})$ on this subspace is determined by the action of $B_{\chi}(\vec{k})$ on the $\mathbb{C}^{|\digits|}$-component. Now suppose the almost sure value of the Lyapunov exponent $\lambda^{B_{\chi}}<\log\sqrt{|\digits|}$. This implies that, for all nonzero starting vector $v\in\mathbb{C}^{|\digits|}$, $\|vB^{(n)}_{\chi}(\vec{k})\|$ 
grows exponentially fast. This contradicts the translation boundedness of $h(\vec{k})$ unless $h(\vec{k})=0$ for Lebesgue-a.e. $\vec{k}$; compare \cite[Thm.~3.28]{BGM}. From this, one gets $\widehat{\gamma}^{ }_{\textsf{ac}}=0$ and hence the diffraction, and equivalently, $\sigma_f$ must be singular, which completes the proof.
\end{proof}

When $\chi(W)$ is neither of full rank, nor rank-$1$, one can still construct the Fourier cocycle for $\sub^{(\chi)}$ and compute its Lyapunov exponent. If it satisfies the singularity condition in Theorem~\ref{thm: absence of ac}, one can invoke Proposition~\ref{prop: factors} to conclude that $H^{\chi}$ only admits singular spectral measures. 

\section{Examples}\label{sec:examples}

\subsection{Planar example with all spectral types}
\begin{example}[Vierdrachen substitution]\label{ex:vierdrachen}

Let $Q=\begin{pmatrix}
1 & -1\\
1 & 1
\end{pmatrix}$ with the associated digit set $\digits=\left\{(0,0),(1,0)\right\}=\left\{\dlet_0,\dlet_1\right\}$. 
Choose the group of spins to be the Klein-4 group $G=C_2\times C_2=\left\{e,a,b,ab\right\}$ and the spin matrix to be $W=\begin{pmatrix}
e & a\\
e & ab
\end{pmatrix}$. We call the qubit spin substitution $\sub$ with these defining data the \emph{Vierdrachen} substitution, alluding to the geometry of the twin dragon tiling \cite{Vince} and the spin group being the Vierergruppe. 

This table shows the colors representing the alphabet:

\centerline{\includegraphics[scale=.15]{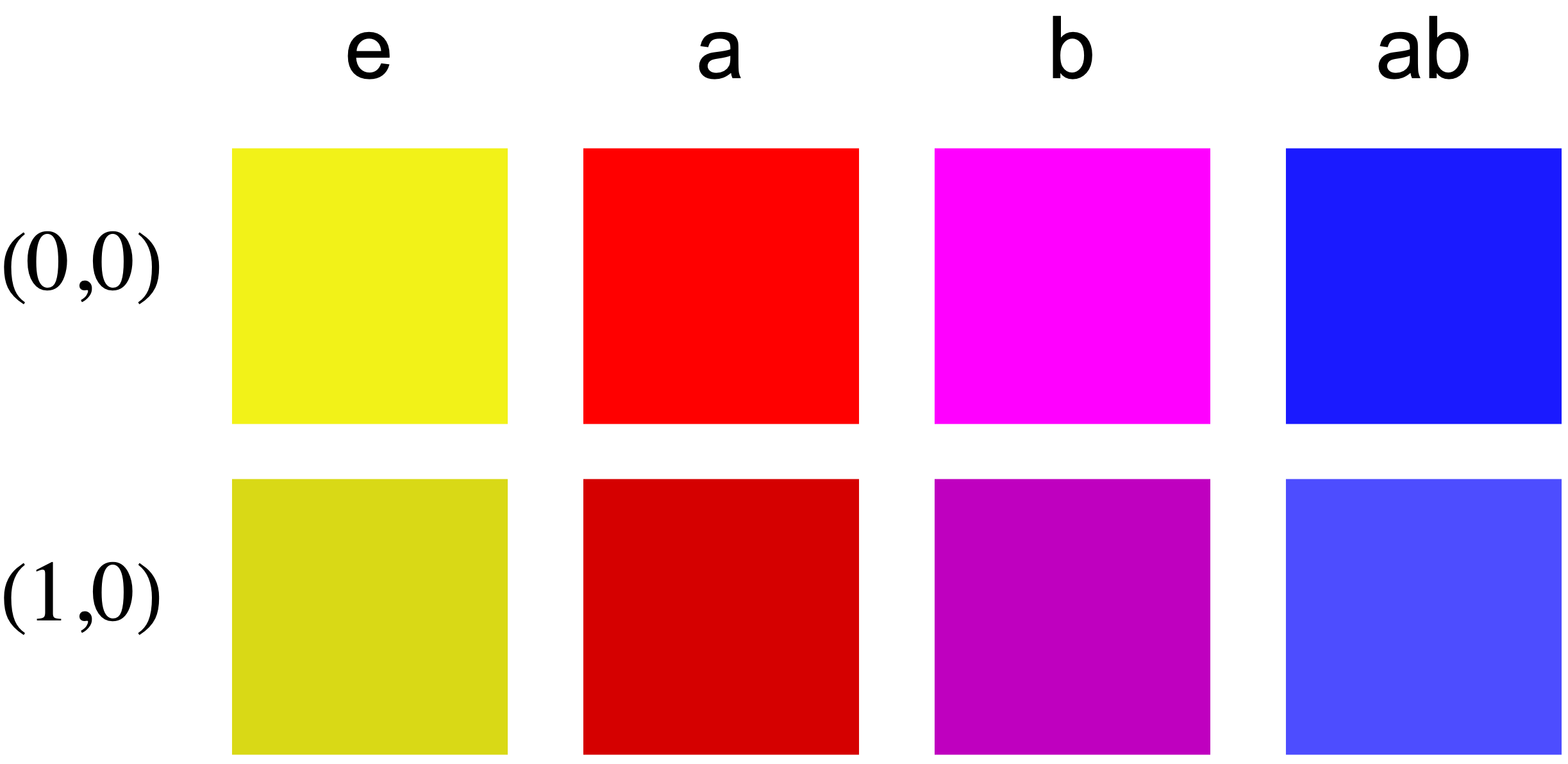}}
\noindent Figure \ref{fig:Viersupertiles} shows the 1- and 2-supertiles organized in the same way.

\begin{figure}[h!]
\includegraphics[scale=.15]{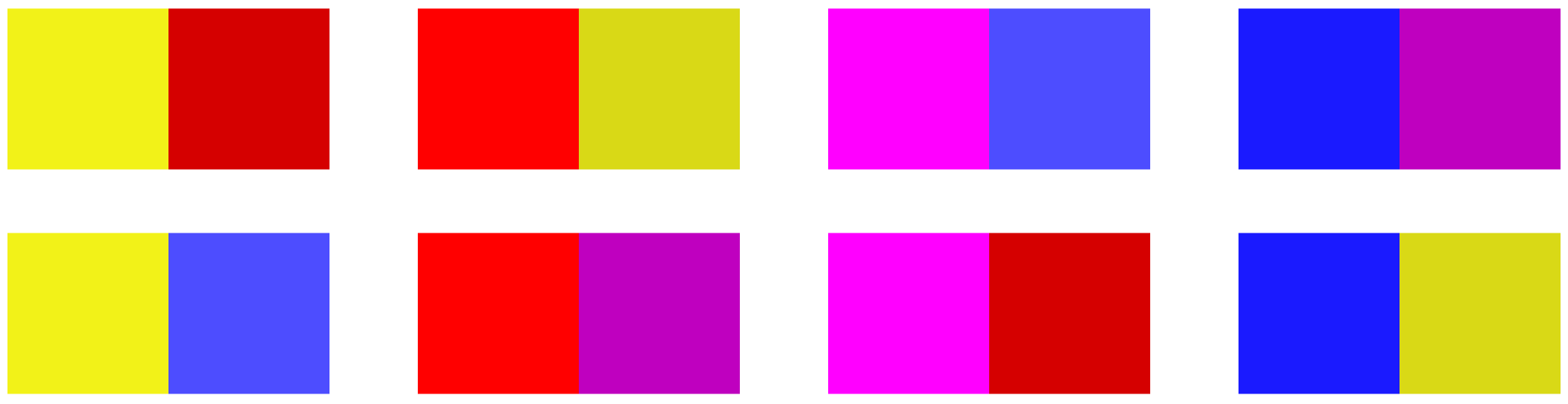}\\
\vskip 2em
\includegraphics[scale=.225]{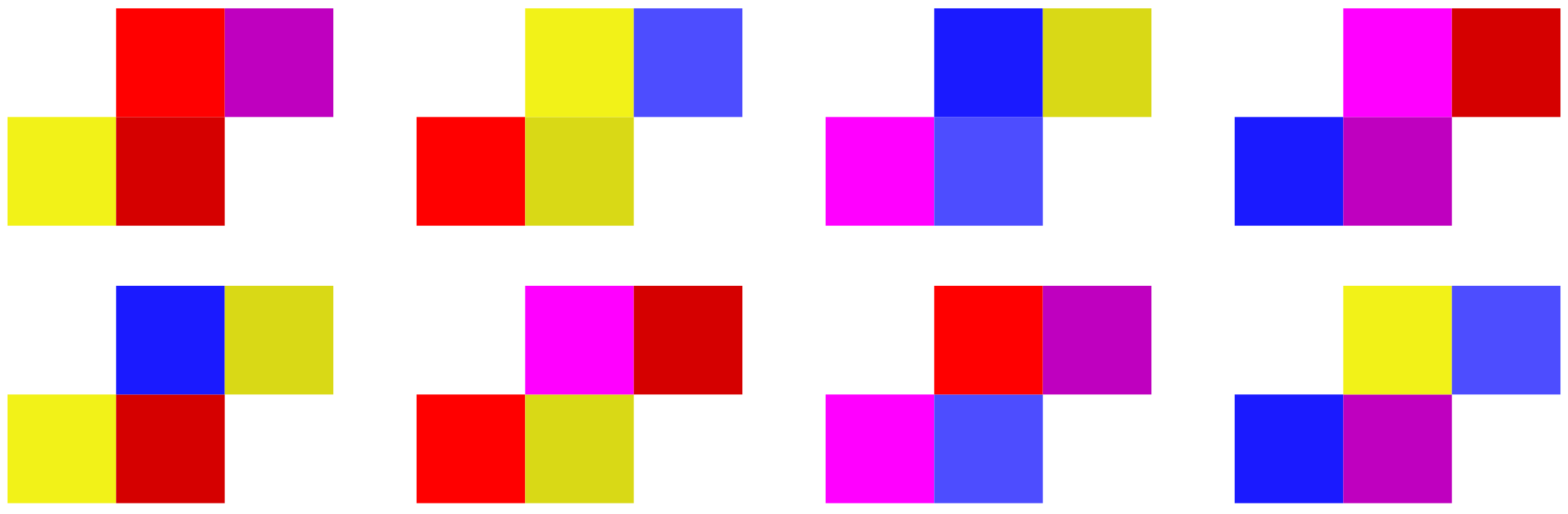}
\caption{The level -1 and -2 supertiles of the Vierdrachen substitution.}
\label{fig:Viersupertiles}
\end{figure}

The left of Figure \ref{fig:Vierbigtile} shows the 7-supertile of type $\dlet_0$, and the right shows its image under the factor map taking a letter to its spin. This ``forget the digits" map, which is a single-block code on $\seqsp$, likely represents a topological conjugacy for many spin substitutions. The argument would be adapted from the proof of equivalence of the Rudin--Shapiro substitution space and the original RS sequence space on $\mA = \{-1, ~1\}$.

\begin{figure}[h!]
\includegraphics[scale=.5]{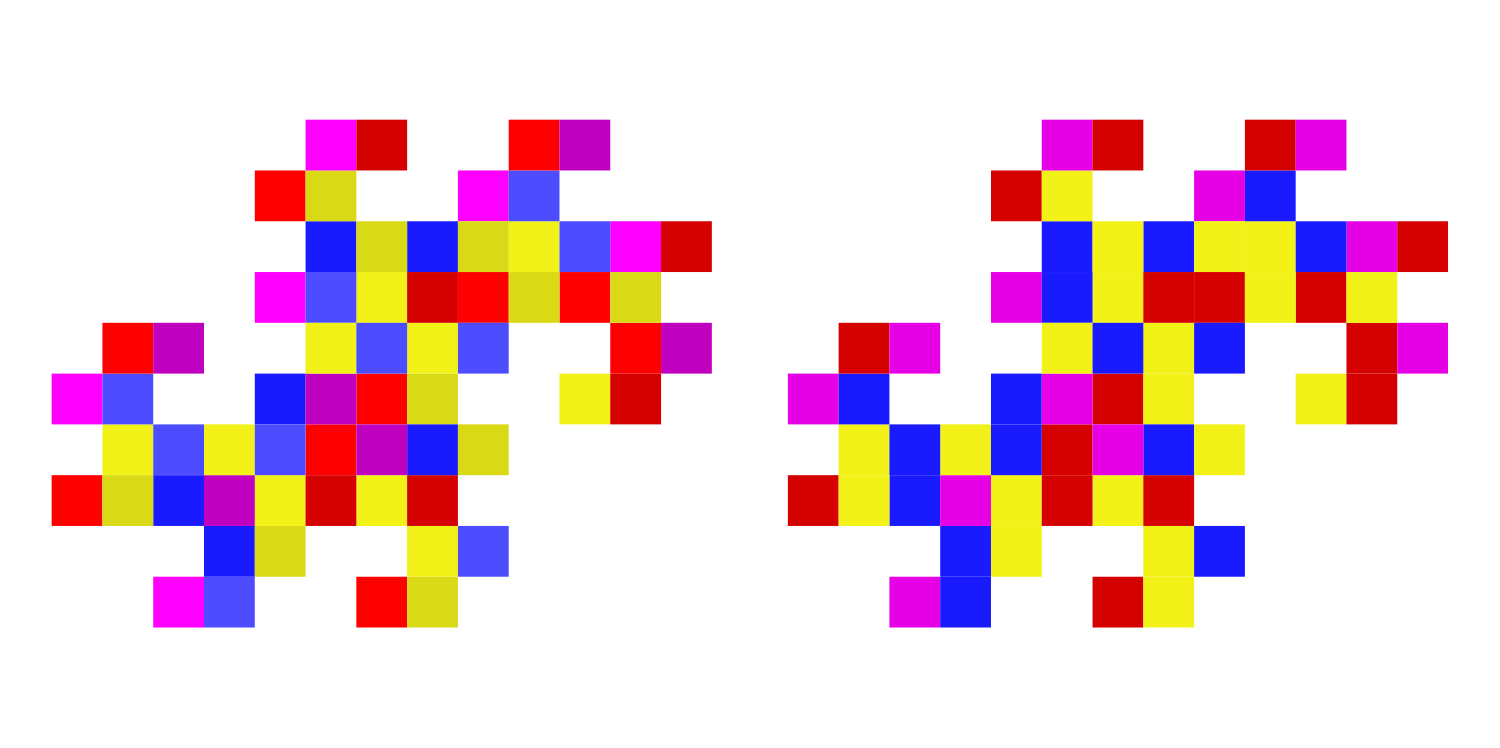}
\caption{The level-7 supertile for $\dlet_0$ and its image under the forget-the-digits map.}\label{fig:Vierbigtile}
\end{figure}

Let $\Gdual=\left\{\chi_0, \chi^{ }_1,\chi^{ }_2,\chi^{ }_3\right\}$, where the corresponding kernels for $\chi_i,0\leqslant i\leqslant 3$ are $G, \left\langle a\right\rangle,\left\langle b\right\rangle$ and $\left\langle ab \right\rangle$, respectively. The corresponding matrices are then 
\[
\chi^{ }_0(W)=\begin{pmatrix}
1 & 1\\
1 & 1 
\end{pmatrix}  \quad
\chi^{ }_1(W)=\begin{pmatrix}
1 & 1\\
1 & -1 
\end{pmatrix} \quad
\chi^{ }_2(W)=\begin{pmatrix}
1 & -1\\
1 & -1 
\end{pmatrix}  \quad
\chi^{ }_3(W)=\begin{pmatrix}
1 & -1\\
1 & 1 
\end{pmatrix}. 
\]
From Theorem~\ref{thm:main result}(a), $\chi_0$ corresponds to the pure point component of the spectrum $H^{\chi_0}$ arising from the odometer. 
The factor substitution for $\chi_0$ is periodic, and a large supertile is shown on the far left of Figure~\ref{fig:vierdrachen}.
It is easy to check that $\frac{1}{\sqrt{2}}\chi_1(W)$ and $\frac{1}{\sqrt{2}}\chi_3(W)$ are unitary, which means $H^{\chi_1}$ and $H^{\chi_2}$ each decompose into two orthogonal cyclic subspaces, all having Lebesgue spectral measure by Theorem~\ref{thm:main result}(b). These two substitutions appear the most disordered of the four in Figure~\ref{fig:vierdrachen}.  Lastly, the matrix $\chi_2(W)$ is rank-1, which means that the factor substitution is bijective. Thus $H^{\chi_3}$ comprises of singular continuous components of multiplicity at most $2$ by Theorem~\ref{thm:main result}(c).

\begin{figure}[h]
\includegraphics[scale=0.375]{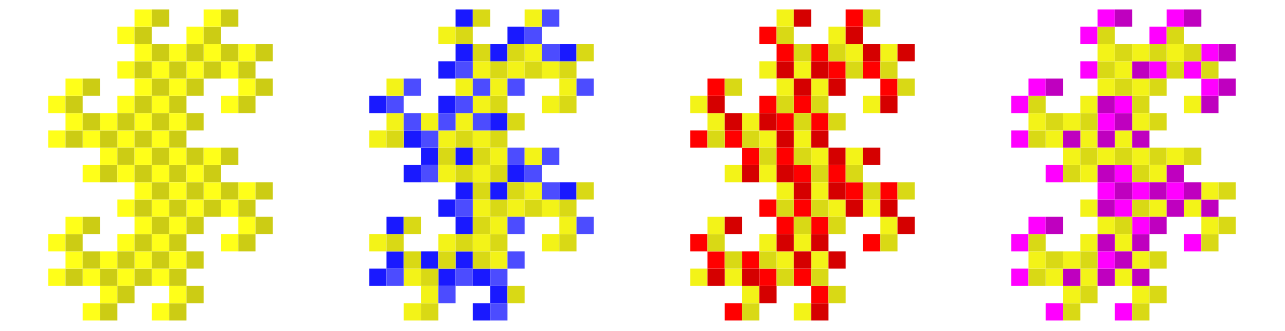}
\caption{The image of $\sub^{8}(\dlet_0)$ under the single-block codes given by each of the characters. 
From (L) to (R), the corresponding character $\chi$ and the spectral type of $H^{\chi}$: $\chi^{ }_0$ (\textsf{pp}), $\chi^{ }_1$ (\textsf{ac}), $\chi^{ }_2$ (\textsf{sc}) and $\chi^{ }_3$ (\textsf{ac}).}\label{fig:vierdrachen}
\end{figure}

This is the planar analog of the one-dimensional substitution in \cite{BGG}, which has the same spin matrix as the Vierdrachen and has $Q=2$ and $\digits=\left\{0,1\right\}$. 

\end{example}

\subsection{Aperiodic example with a non-trivial periodic factor}

\begin{example}\label{ex:pp-factor}
Let $Q=3$, $\mathcal{D}=\left\{0,1,2\right\}$ and $G=C_4=\left\langle\ii \right\rangle$. Consider the spin matrix 
\[
W=\begin{pmatrix}
1 & \ii & 1\\
1 & -\ii & -1\\
1 & \ii & -1
\end{pmatrix}. 
\]
One can verify that the corresponding spin substitution $\sub$ is primitive and aperiodic via the substitution matrix and the existence of non-trivial proximal pairs, respectively. Moreover, $\chi_2(W)$ is a rank-1 matrix, where $\chi_2\in \Gdual$ with $\chi_2(g)=g^2$ for $g\in C_4$. From Theorem~\ref{thm:main result}, $\chi_2$ produces a factor substitution $\sub^{(\chi_2)}$. Here, $\sub^{(\chi_2)}\colon \alet\mapsto \alet\blet\alet, \blet\mapsto \blet\alet\blet$, which is periodic. This means that the set of eigenvalues contains $\ZZ[\frac{1}{3}]\times C_2$. One can also check that the substitution matrix $M$ does not have any eigenvalue of modulus $\sqrt{3}$, and hence the spectrum of $L^{2}(\seqsp,\mu)$ is purely singular by Theorem~\ref{thm: Berlinkov-Solomyak}. There are also analogues of this example in higher dimensions. \exend
\end{example}

\subsection{Qubit spin substitutions from Kronecker products}
It is natural to ask whether one can build a new qubit spin substitution using two old ones by combining their spin matrices.
 Given two qubit substitutions $\mathcal{S}_1=(\mathcal{D}_1,W_1)$ and $\mathcal{S}_2=(\mathcal{D}_2,W_2)$ with underlying spin groups $G_1$ and $G_2$, respectively, one can  define $\mathcal{S}:=(\mathcal{D},W_1\otimes W_2)$, where $|\mathcal{D}|=|\mathcal{D}_1|\cdot|\mathcal{D}_2|$ with spin group $G=G_1\times G_2$ as follows
\begin{enumerate}
\item Index the elements of the new digit set $\mathcal{D}$ via elements of $\mathcal{D}_1\otimes \mathcal{D}_2$, i.e., $\vec{d}=\vec{d}_1\otimes \vec{d}_2$ and order them using the lexicographic order in $\mathcal{D}_1\otimes\mathcal{D}_2$.
\item For $\vec{d},\vec{d}^{\prime}\in \mathcal{D}$, define $W(\vec{d},\vec{d}^{\prime}):=(W_1(\vec{d}^{ }_1,\vec{d}^{\prime}_1),W_2(\vec{d}^{ }_2,\vec{d}^{\prime}_2))\in G_1\times G_2$.  
\end{enumerate}

Some comments on the notation are in order. 
We emphasize that the tensor product $\vec{d}_1\otimes \vec{d}_2$ does not have any spatial interpretation whatsover. It is only used as bookkeeping in order to assign the right spin in the supertiles of the new digit set, which is done in the second line above.
We also point out that this substitution is not unique. In fact, our only restriction is on the cardinality of $\mathcal{D}$, and we have the freedom to choose the expansive map $Q$, the digit set $\mathcal{D}$ itself and  even the dimension of the space where the substitution is defined. 
For such an $\mathcal{S}$, one has the following properties.
\begin{prop}
Let $\mathcal{S}_1$ and $\mathcal{S}_2$ be defined as above and $G=G_1\times G_2$ be the underlying spin group. Let $\mathcal{S}=(\mathcal{D},W_1\otimes W_2)$ be a qubit substitution arising from $\mathcal{S}_1$ and $\mathcal{S}_2$ via the construction above. Then, the following holds. 
\begin{enumerate}
\item $\mathcal{S}$ is primitive if and only if $\mathcal{S}_1$ and $\mathcal{S}_2$ are both primitive. 
\item Let $\chi=\chi_1\otimes \chi_2\in \widehat{G_1\times G_2}$. Then, 
$\sub$ is $\chi$-unitary if and only 
$\sub_1$ and $\sub_2$ are $\chi_1$- and $\chi_2$-unitary, respectively. Similarly, 
$\sub$ is $\chi$-rank-$1$ if and only if  $\sub_1$ and $\sub_2$ are rank-$1$ with respect to $\chi_1$ and $\chi_2$.
\end{enumerate}
\end{prop}

\begin{proof}
The first claim follows from the fact that the substitution matrix of $\mathcal{S}$ inherits the Kronecker product structure, i.e.,  $M_{\mathcal{S}}=M_{\mathcal{S}_1}\otimes M_{\mathcal{S}_2}$, hence if one chooses $n=\text{lcm}(n_1,n_2)$, where $n_i$ is the index of primitivity of $M_{\mathcal{S}_i}$, one gets that $M^n_{\mathcal{S}}=M^n_{\mathcal{S}_1}\otimes M^n_{\mathcal{S}_2}$ is a positive matrix. 
The second claim follows from the fact that Kronecker products of unitary matrices are unitary and that rank is multiplicative under $\otimes$. 
\end{proof}

\subsection{Singular subspaces via Lyapunov exponents}

\begin{example}
Let $G=C_4=\left\langle\ii\right\rangle $ and $\digits =\left\{0,1\right\}$. 
Consider the map 
\[
W=\begin{pmatrix}
1 & \ii\\
1 & -1
\end{pmatrix}.
\]
Let $\chi_n\colon g\mapsto g^n$ be a character in $\widehat{C_4}$. For $n=0$, we get the odometer factor.  For $n=2$, we get that $\frac{1}{\sqrt{2}}\chi_2(W)$ is unitary, so $H^{\chi_2}$ is of Lebesgue type. For $n=1,3$, we do not get smaller substitutive factors since $\text{ker}(\chi_1)=\text{ker}(\chi_3)=\left\{1\right\}$, but $\chi_1(W)$ and $\chi_3(W)$ are both of full rank, which means we can use Lyapunov exponents to prove the singularity of the spectrum by invoking Proposition~\ref{prop:Lyapunov-spin-singular}. The corresponding (complex-valued) matrix cocycles are
\[
B_{\chi^{ }_1}(\vec{k})=\begin{pmatrix}
1 & \ii\\
\ee^{2\pi\ii\vec{k}} & -\ee^{2\pi\ii\vec{k}}
\end{pmatrix} \quad\text{ and }\quad 
B_{\chi^{ }_3}(\vec{k})=\begin{pmatrix}
1 & -\ii\\
\ee^{2\pi\ii\vec{k}} & -\ee^{2\pi\ii\vec{k}}
\end{pmatrix}
\]
A sequence of almost sure upper bounds of the Lyapunov exponent is given by 
\begin{equation}\label{eq:Lyapunov-UB}
f(N)=\frac{1}{2N}\int_0^1 \log\|B^{(N)}_{\chi}(\vec{x})\|^2\dd \vec{x}  
\end{equation}
which can be computed numerically for each $N\in\mathbb{N}$. The expansion factor for this example is $|\det(Q)|=|\mathcal{D}|=2$. If, for some $N$, this upper bound in Eq.~\eqref{eq:Lyapunov-UB} is strictly less than $\log\sqrt{2}$, one gets that $\lambda^{B_{\chi}}(\vec{k})<\log\sqrt{2}$ for a.e. $\vec{k}\in\mathbb{R}$; compare \cite{BFGR,BGM}.

\begin{table}[h]
\begin{center}
\renewcommand{\arraystretch}{1.5}
\begin{tabular}{|c|c|c|c|c|}
\hline 
$N$ & 10 & 11 & 12 & 13    \\
\hline  
$\frac{1}{N}\int_{\mathbb{T}}\log\|B_{\chi^{ }_1}^{(N)}(\vec{x})\|^{2}_{\text{F}}$ &  0.703953 & 0.695342 & 0.688005 & 0.682035   \\
\hline 
\end{tabular}
\vspace{3mm}
\caption{Numerical values for upper bounds for $2f(N)$ for $B_{\chi^{ }_1}$. Here $\|\cdot\|^{ }_{\text{F}}$ stands for the Frobenius norm. All numerical errors are less that $10^{-3}$.}\label{tab:Lyapunov-C4}

\end{center}
\end{table}

Since $\log(2)\approx 0.693147$, we get that $\lambda^{B_{\chi}}(\vec{k})\leqslant f(12)<\log\sqrt{2}$ for a.e. $\vec{k}\in\R$, which implies $H^{\chi_1}$ is singular. One can carry out the same computation for $\chi_3$ which leads to the same result.

\end{example}

\section{Summary and Outlook}

We can summarize the spectral analysis we have for qubit spin substitutions as follows. 
Proposition~\ref{prop: L2 decomp} tells us that $L^{2}(X,\mu)$ splits into a direct sum of orthogonal subspaces $H^{\chi}$ each corresponding to a $\chi\in\Gdual$. To analyze the spectral properties of $H^{\chi} $ we ask: 
\begin{enumerate}
\item Is $\chi(W)$ of full rank? 
\begin{enumerate}
\item Is $\frac{1}{|\digits|}\chi(W)$ unitary? If yes, $H^{\chi}$ has $|\mathcal{D}|$ orthogonal absolutely continuous components.  (Theorem~\ref{thm:main result}(b))
\item If $\frac{1}{|\digits|}\chi(W)$ is of full rank but not unitary, use Lyapunov exponents to confirm singularity (Proposition~\ref{prop: LE full rank})
\end{enumerate}
\item Is $\chi(W)$ of rank 1? 
\begin{enumerate}
\item If yes, then $H^{\chi}$ is singular with at most $|G/\text{ker}(\chi)|$ orthogonal components. (Theorem~\ref{thm:main result}(c))
\item If $\sub^{\prime}$ is aperiodic, then $H^{\chi}$ must be purely singular continuous.
\end{enumerate}
\item Is $\chi(W)$ of rank between $1$ and $|\digits|$?
\begin{enumerate}
\item Check whether the factor substitution
$\sub^{(\chi)}$ has singular spectrum (e.g. via Lyapunov exponents). If it does, then $H^{\chi}$ must be singular. (Proposition~\ref{prop: factors})
\end{enumerate}

\end{enumerate}

\begin{question}\label{ques:ac}
Let $\mathcal{S}$ be a primitive qubit substitution arising from a digit set $\mathcal{D}$ and group $G$. Is it true that the following are equivalent?
\begin{enumerate}
\item There are exactly $n$ characters $\chi\in \Gdual$ for which $\mathcal{S}$ is $\chi$-unitary.
\item The substitution matrix $M$ has exactly $n|\digits|$ eigenvalues of modulus $\sqrt{|\digits|}$. 
\item The Fourier matrix $B(\vec{k})$ has Lyapunov exponent $\log\sqrt{|\digits|}$ of multiplicity $n|\digits|$. 
\item $L^{2}(\seqsp,\mu)$ has a Lebesgue component of multiplicity $n|\digits|$. 
\end{enumerate}
Equivalently, is there any other mechanism for $H^{\chi}$ to be absolutely continuous apart from $\chi$-unitarity? 
\end{question}

\begin{question}
All of the examples we considered have even Lebesgue multiplicity.
Is there a primitive qubit substitution $\mathcal{S}$ whose Lebesgue multiplicity is odd? This requires $|\mathcal{D}|$ to be odd if the equivalences in Question~\ref{ques:ac} hold. 
\end{question}

\section*{Acknowledgements}

The authors would like to thank Michael Baake for valuable suggestions and comments on the manuscript. We also express our gratitude to Franz G\"ahler, Uwe Grimm, Daniel Lenz, May Mei and Lorenzo Sadun for fruitful discussions. NM is funded by the German Research Foundation (DFG, Deutsche
Forschungsgemeinschaft), via SFB 1283/2 2021--317210226.

\bibliographystyle{alpha}

\begin{thebibliography}{XXXXXXX}

\bibitem[AL13]{AL}
Abou I and Liardet P, Flots cha\^{i}n\'{e}s, In \textit{Proceedings of the Sixth Congress of Romanian Mathematicians Vol.1}, L. Beznea, V. Brinzanescu, R. Purice, et.al. (eds.), Editura Academiei Rom\^{a}ne, Bucharest (2009) pp.~401--432. 

\bibitem[AL91]{AL-2}
Allouche J P and Liardet P, Generalized Rudin--Shapiro sequences, \textit{Acta Arith.} \textbf{60} (1991) 1--27. 

\bibitem[BFGR19]{BFGR}
Baake M, Frank N P, Grimm U and Robinson E A,
Geometric properties of a binary non-Pisot inflation
and absence of absolutely continuous diffraction,
\textit{Studia Math.} \textbf{247} (2019) 109--154.
 

\bibitem[BGG13]{BGG}
Baake M, G\"ahler F and Grimm U, Examples of substitution systems and their factors, \textit{J. Int. Seq.} \textbf{16} (2013) 13.2.14 (18pp).

\bibitem[BG\"aM19]{BGM}
Baake M, G\"{a}hler F and Ma\~{n}ibo N, Renormalisation of pair correlation measures for primitive inflation rules and absence of absolutely continuous diffraction, \textit{Commun. Math. Phys.} \textbf{370} (2019) 591--635.

\bibitem[BG13]{BG}
Baake M and Grimm U, \textit{Aperiodic Order Volume 1. A Mathematical Invitation}, Cambridge University Press, Cambridge (2013). 

\bibitem[BG14]{BG2}
Baake M and Grimm U, Squirals and beyond: substitution tilings with singular continuous spectrum, \textit{Ergodic Th. \& Dynam. Syst.} \textbf{34} (2014) 1077--1102.

\bibitem[BL04]{BL}
Baake M and Lenz D, Dynamical systems on translation bounded measures: pure point dynamical and diffraction spectra, \textit{Ergod. Th. \& Dynam. Syst.} \textbf{24} (2004) 1867--1893. 


\bibitem[BLvE15]{BLvE}
Baake M, Lenz D and van Enter A, Dynamical versus diffraction spectrum for structures with finite local complexity, \textit{Ergod. Th. \& Dynam. Syst.} \textbf{35} (2015) 2017--2043. 


\bibitem[BP07]{BP}
Barreira L and Pesin Y, \textit{Nonuniform Hyperbolicity}, Cambridge University Press, Cambridge (2007).

\bibitem[Bar18]{Bartlett}
Bartlett A, Spectral theory of $\mathbb{Z}^d$ substitutions, \textit{Ergodic Th. \& Dynam. Syst.} \textbf{38} (2018) 1289--1341. 


\bibitem[BS19]{BS}
Berlinkov A and Solomyak B, Singular substitutions of constant length,
\textit{Ergodic Th. \& Dynam. Syst.} \textbf{39} (2019) 2384--2402.

\bibitem[BS20]{BuS}
Bufetov A and Solomyak B, A spectral cocycle for substitution systems and translation flows, \textit{J. Anal. Math.} \textbf{141} (2020) 165--205.   


\bibitem[Cab21]{Cabezas} 
Cabezas C, Homomorphisms between multidimensional constant-shape substitutions, \textit{preprint} \texttt{arXiv:2106.10504}.


\bibitem[CGS18]{CGS}
Chan L, Grimm U and Short I, Substitution-based structures with absolutely continuous spectrum, \textit{Indag. Math.} \textbf{29} (2018) 1072--1086. 


\bibitem[Cor06]{Cortez}
Cortez, M I, $\ZZ^d$ Toeplitz arrays, \textit{Discr. Contin. Dynam. Syst. A} \textbf{15} (2006) 859--881.

\bibitem[CM99]{CM}
Coven E and Meyerowitz A, Tiling the integers with translates of one finite set, \textit{J. Algebra} \textbf{212} (1999) 161--174. 

\bibitem[EL11]{EL}
el Abdalaoui E H and Lema\'{n}czyk M, Approximately transitive dynamical systems and simple spectrum, \textit{Arch. Math.}  \textbf{97} (2011) 187--197.



\bibitem[Fra20]{Frank2}
Frank N P, Introduction to hierarchical tiling dynamical systems, In \textit{Substitution and Tiling Dynamics: Introduction to Self-inducing Structures},
S. Akiyama and P. Arnoux (eds.),
LNM 2773, Springer, Cham (2020), pp.~33--95.

\bibitem[Fra05]{Frank3}
Frank N P, Multidimensional constant-length substitution sequences, \textit{Topology \& Appl.}
\textbf{152} (2005) 44--69.

\bibitem[Fra03]{Frank1}
Frank N P, Substitution sequences in $\mathbb{Z}^d$ with a nonsimple Lebesgue component in the spectrum, \textit{Ergodic Th. \& Dynam. Syst. } \textbf{23} (2003) 519--532.

\bibitem[FS14a]{FS1}  Frank N P and Sadun L, Fusion: a general framework for hierarchical tilings of $\mathbb{R}^d$, \textit{Geom. Dedicata} \textbf{171} (2014) 149--186. 

\bibitem[FS14b]{FS2} Frank N P and Sadun L, Fusion tilings with infinite local complexity, \textit{Top. Proc.} \textbf{43} (2014) 235--276. 

\bibitem[Fer07]{Fer}
Fernique T, Local rule substitutions and stepped surfaces, \textit{Theoret. Comp. Sci.} \textbf{380} (2007) 317--329. 


\bibitem[GT20]{GT}
Greenfeld R and Tao T, The structure of translational tilings in $\mathbb{Z}^d$, \textit{preprint} (2020) \texttt{arXiv:2010.03254}.  

\bibitem[GH94]{GroHaas}
Gr\"ochenig K and Haas A, Self-similar lattice tilings, \textit{J. Fourier Anal.} \textbf{1} (1994) 131--170.


\bibitem[Hel86]{Helson}
Helson H, Cocycles on the circle, \textit{J. Oper. Theory} \textbf{16} (1986) 189--199.


\bibitem[Ken95]{Kenyon}
Kenyon R, Self-replicating tilings, In 
\textit{Symbolic Dynamics and its Applications}, P. Walters (ed.),
\textit{Contemp. Math.} 
 \textbf{135} (1992) 239--263.


\bibitem[LW96a]{LW-2}
Lagarias J C and Wang Y, Integral self-affine tiles in $\mathbb{R}^n$ I. standard and nonstandard digit sets,
\textit{J. London Math. Soc.} \textbf{54} (1996) 161--179.  


\bibitem[LW96b]{LW}
Lagarias J C and Wang Y, Self-affine tiles in $\mathbb{R}^n$, \textit{Adv. Math.} \textbf{121} (1996) 21--49. 


\bibitem[LMS03]{LMS-2}
Lee J-Y, Moody R V and Solomyak B, 
Consequences of pure point diffraction spectra for multiset substitution systems, \textit{Discrete Comput. Geom.}
\textbf{29} (2003) 525--560. 



\bibitem[LMS02]{LMS}
Lee J-Y, Moody R V and Solomyak B, Pure point dynamical and diffraction spectra, \textit{Ann. Henri Poincar\'{e}} \textbf{3} 1003--1018. 


\bibitem[Len20]{Lenz}
Lenz D, Spectral theory of dynamical systems as diffraction theory of sampling functions, \textit{Monats. Math.} \textbf{192} (2020) 625--649.

\bibitem[Man17]{Manibo}
Ma\~nibo N, Lyapunov exponents for binary constant-length substitutions, \textit{J. Math. Phys.} \textbf{58} (2017) 113504 (9pp).

\bibitem[MRW21]{MRW}
Ma\~nibo N, Rust D and Walton J, Spectral properties of substitutions on compact alphabets, \textit{preprint}, \texttt{arXiv:2108.01762}.


\bibitem[Mes95]{Meshulam}
Meshulam R, On subsets of finite abelian groups with no $3$-term arithmetic progression, \textit{J. Combin. Theor. A} \textbf{71} (1995) 168--172.

\bibitem[Nad98]{Nadkarni}
Nadkarni M G, The skew product, In \textit{Spectral Theory of Dynamical Systems}, R.B. Bapat, V.S. Borkar, P. Chaudhuri, et.al. (eds.), Hindustan Book Agency, Gurgaon (1998) pp.~37--39.

\bibitem[Que10]{Queffelec2}
Queff\'{e}lec  M, Substitution Dynamical Systems--Spectral Analysis, 2nd.~ed., LNM 1294, Springer, Berlin (2010). 


\bibitem[Que87]{Queffelec1}
Queff\'{e}lec  M, Une nouvelle properi\'{e}t\'{e} des suites de Rudin--Shapiro, \textit{Ann. Inst. Fourier} \textbf{37} (1987) 115--138. 


\bibitem[Rob88]{Robbie}
Robinson E A, Non-abelian extensions have nonsimple spectrum, \textit{Compos. Math.} \textbf{65} (1988) 155--170.

\bibitem[Sol97]{Solomyak}
Solomyak B, Dynamics of self-similar tilings, \textit{Ergodic Th. \& Dynam. Syst.} \textbf{17} (1997) 695--738 and \textit{Ergodic Th. \& Dynam. Syst.} \textbf{19} (1999) 1685 (erratum). 


\bibitem[Sol98]{Solomyak2}
Solomyak B, Nonperiodicity implies unique composition for self-similar translationally finite tilings, \textit{Discrete Comput. Geom.} \textbf{20} (1998) 265--279.

\bibitem[Via13]{Viana}
Viana M, \textit{Lectures on Lyapunov Exponents}, Cambridge University Press, Cambridge, (2013). 

\bibitem[Vin00]{Vince}
Vince A, Digit tiling of Euclidean space, In \textit{Directions in Mathematical Quasicrystals}, M. Baake and R.V. Moody (eds.), AMS, Providence, RI (2000), pp.~329--370.  


\bibitem[Vin95]{Vince2}
Vince A, Rep-tiling Euclidean space, \textit{Aequationes Math.} \textbf{50} (1995) 191--213.


\end{thebibliography}

\end{document}